\documentclass[12pt, twoside]{article}
\usepackage[T1, OT1]{fontenc}
\usepackage{hyperref}
\usepackage{authblk, lipsum}
\usepackage{amssymb}
\usepackage{amsmath, enumerate, mathabx, tikz}
\usepackage{amsthm}
\usepackage{color}
\usepackage{mathrsfs}
\usepackage{txfonts}

\usetikzlibrary{positioning, patterns}
\usetikzlibrary{shapes.geometric,arrows}
\tikzstyle{box} = [rectangle,text centered, draw=black]
\tikzstyle{arrow} = [thick, ->, >=stealth]
\usepackage{float} 
\usepackage{enumerate}
\usepackage{mathrsfs}

\allowdisplaybreaks

\newtheorem{theorem}{Theorem}[section]
\newtheorem{lemma}[theorem]{Lemma}
\newtheorem{proposition}[theorem]{Proposition}
\newtheorem{corollary}[theorem]{Corollary}
\theoremstyle{definition}

\numberwithin{equation}{section}

\begin{document}
\UseRawInputEncoding
\arraycolsep=1pt

\title{\bf On odd-normal numbers   
\footnotetext{\hspace{-0.35cm} 2020 {\it
Mathematics Subject Classification}.
Primary 42A63, 28A78, 11K16; Secondary 42A38, 28A80, 11K36.
\endgraf {\it Key words and phrases}.
Fourier analysis, Fourier series, trigonometric series,  sets of uniqueness and of multiplicity, normal and non-normal numbers, Rajchman measures, metric number theory 
\endgraf
}}
\author{Malabika Pramanik and Junqiang Zhang}  
\date{ }
\maketitle
\newcommand{\Addresses}{{
  \bigskip
  \footnotesize

  J.~Zhang, \textsc{School of Science, China University of Mining and Technology-Beijing, Beijing 100083, P.~R.~China}\par\nopagebreak
  \textit{E-mail address}: \texttt{jqzhang@cumtb.edu.cn}

  \medskip

  M.~Pramanik, \textsc{Department of Mathematics, 1984 Mathematics Road, University of British Columbia, Vancouver, Canada V6T 1Z2} \par\nopagebreak
  \textit{E-mail address}: \texttt{malabika@math.ubc.ca}

%

}}

\vspace{-0.8cm}

\begin{center}
\begin{minipage}{13cm}
{\small {\bf Abstract.} A real number $x$ is considered normal in an integer base $b \geq 2$ if its digit expansion in this base is ``equitable'', ensuring that for each $k \geq 1$, every ordered sequence of $k$ digits from $\{0, 1, \ldots, b-1\}$ occurs in the digit expansion of $x$ with the same limiting frequency. Borel's classical result \cite{b09} asserts that Lebesgue-almost every $x \in \mathbb R$ is normal in every base $b \geq 2$. 
\vskip0.1in 
\noindent This paper serves as a case study of the measure-theoretic properties of Lebesgue-null sets containing numbers that are normal only in certain bases.  We consider the set $\mathscr N(\mathscr{O}, \mathscr{E})$ of reals that are normal in odd bases but not in even ones. This set has full Hausdorff dimension \cite{p81} but zero Fourier dimension. The latter condition means that $\mathscr N(\mathscr{O}, \mathscr{E})$ cannot support a probability measure whose Fourier transform has power decay at infinity. Our main result is that $\mathscr N(\mathscr{O}, \mathscr{E})$ supports a Rajchman measure $\mu$, whose Fourier transform $\widehat{\mu}(\xi)$ approaches 0 as $|\xi| \rightarrow \infty$ by definiton, albeit slower than any negative power of $|\xi|$. Moreover, the decay rate of $\widehat{\mu}$ is essentially optimal, subject to the constraints of its support. The methods draw inspiration from the number-theoretic results of Schmidt \cite{s60} and a construction of Lyons \cite{l86}.
\vskip0.1in
\noindent  As a consequence, $\mathscr N(\mathscr{O}, \mathscr{E})$ emerges as a set of multiplicity, in the sense of Fourier analysis. This addresses a question posed by Kahane and Salem \cite{Kahane-Salem-64} in the special case of $\mathscr N(\mathscr{O}, \mathscr{E})$.
}    
%
\end{minipage}
\end{center}

\tableofcontents
\section{Introduction}
Given a number $x \in [0,1)$ and an integer $b \in \{2, 3, \ldots \} = \mathbb N \setminus \{1\}$, let $x = 0.x_1x_2 \cdots$ denote the digit expansion of $x$ with respect to base $b$. We say that $x$ is {\em{normal}} in this base (or {\em{$b$-normal}} for short) if for any $k \geq 1$, every ordered sequence of $k$ digits from $\{0, 1, \ldots, b-1\}$ occurs equally often in its digit expansion; specifically, for every block of digits $(d_1, \ldots, d_k) \in \{0, 1, \ldots, b-1\}^k$,  
  \[ \lim_{N \rightarrow \infty} \frac{1}{N} \#\Bigl\{1 \leq n \leq N : (x_n, x_{n+1}, \ldots, x_{n+k-1}) = (d_1, \ldots, d_k) \Bigr\} = \frac{1}{b^k}. \] 
The research area surrounding normal numbers is vast and diverse. Different perspectives are reflected in Koksma \cite{Koksma-book}, Niven \cite{Niven-book}, Salem \cite{Salem-book}, Kuipers and Niederreiter \cite{KN-book},  Bugeaud \cite{b12}, and in the bibliography therein. This article focuses on a set of limited normality, in the sense described below. 
\vskip0.1in 
\noindent A classical theorem of \'E.~Borel \cite{b09} says that Lebesgue-almost every $x \in [0,1)$ is normal in every base $b \geq 2$. In view of the Fourier decay property of Lebesgue absolutely continuous measures on $[0,1)$, Kahane and Salem \cite{Kahane-Salem-64} posed the following question: does a similar statement hold for {\em{any}} Borel measure on $[0, 1)$ whose Fourier coefficients vanish at infinity? Such measures are known in the literature as {\em{Rajchman measures}}. So their question can be reformulated as: 
\begin{equation} \label{KS-q} 
{\text{\em{Can a set of non-normal numbers support a Rajchman measure?}}}
\end{equation} 
Question \eqref{KS-q} was answered in the affirmative by Lyons \cite{l86} for the set of numbers that are not 2-normal. This is equivalent to a negative answer for the original question of Kahane and Salem \cite{Kahane-Salem-64} for this set. But the problem remains open in full generality for sets of numbers that are normal in certain bases, and not normal in others. The sharper version of this question appears in \eqref{KS-q2} below. The main objective of this article is to establish an affirmative answer to the question \eqref{KS-q} for the set of numbers that are odd-normal but not even-normal. 
\vskip0.1in
\noindent Apart from its obvious connection with the theory of measures and numbers, question \eqref{KS-q} is important in Fourier analysis. A set $E \subseteq [0,1)$ is a {\em{set of uniqueness}} if any complex-valued trigonometric series of the form 
\begin{equation*} 
\sum_{n \in \mathbb Z} a_n e(nx), \qquad e(y) := e^{2 \pi i y}
\end{equation*} 
that converges to zero on $[0,1) \setminus E$ must be identically zero, i.e, $a_n = 0$ for all $n \in \mathbb Z$. If $E$ is not a set of uniqueness, it is called a {\em{set of multiplicity}}. With a rich history rooted in the early works of Riemann \cite{Riemann} and Cantor \cite{Cantor} and spanning more than a century, these sets have been studied in a variety of settings \cite{{Young}, {Menshov}, {Salem}, {Zygmund}, {Salem-Zygmund}, {Bary}}. A comprehensive account may be found in the book of Kechris and Louveau  \cite{KL-book}. More recently, spurred by advances in ergodic theory, fractal geometry and metrical number theory, properties of uniqueness have been examined for self-similar fractal sets \cite{{Li-Sahlsten}, {Varju-Yu}, {Bremont}, {Algom-Hertz-Wang}, {Gao-Ma-Song-Zhang}, {Rapaport}} and for number-theoretic sets occurring in Diophantine approximation \cite{{Lyons-thesis}, {Lyons-85}, {l86}, {b12}, {Bluhm}}.  The connection between sets of uniqueness and the Rajchman property is described in \cite[Chapter 2, \S4-5]{KL-book}. The following theorem \cite[Chapter 2, Theorem 4.1]{KL-book}, attributed to the combined work of Piatetski-Shapiro \cite{{P-S1}, {P-S2}} and Kahane and Salem \cite{Kahane-Salem-63}, is particularly relevant for this article: {\em{if $E$ supports a Rajchman measure, then $E$ is a set of multiplicity}}. Thus, an affirmative answer to question \eqref{KS-q} for a set of non-normal numbers would automatically identify it as a set of multiplicity.            

\subsection{Definitions on normality} 
Let us introduce the following notation. For $\mathscr{B}, \mathscr{B}' \subseteq \mathbb N \setminus \{1\}$, we define 
\begin{align}  
\mathscr N(\mathscr B, \cdot) &:=  \Bigl\{x \in \mathbb R: x \text{ is $b$-normal for all } b \in \mathscr B \Bigr\}, \label{N1} \\  
\mathscr N(\cdot, \mathscr B') &:=  \Bigl\{x \in \mathbb R: x \text{ is not $b'$- normal for any } b' \in \mathscr B' \Bigr\}, \label{N2}  \\
\mathscr N(\mathscr B, \mathscr B') &:= \left\{x \in \mathbb R \; \Biggl| \;  \begin{aligned} &x \text{ is normal in base } b \text{ for all } b \in \mathscr B, \\  &x \text{ is non-normal in base } b' \text{ for all } b' \in \mathscr B' \end{aligned}  \right\}. \label{N3} 
\end{align}  
Clearly $\mathscr N(\mathscr B, \mathscr B') = \mathscr N(\mathscr B, \cdot) \cap \mathscr N(\cdot, \mathscr B')$. 
\vskip0.1in 
\noindent The notion of multiplicative independence is important in determining normality with respect to different bases. We say that $r, s \in \mathbb N \setminus \{1\}$ are {\em{multiplicatively dependent}}, and write 
\begin{equation} \label{def-mult-dep} r \sim s \quad \text{ if } \quad \frac{\log r}{\log s} \in \mathbb Q, \; \; \text{that is,} \; \; \text{if there exist $m, n \in \mathbb N$ such that } r^m = s^n.  \end{equation}    
Otherwise, we put $r \not\sim s$ and call $r, s$ multiplicatively independent. Schmidt \cite[Theorem 1]{s60} has shown that if $r \sim s$, then any number that is $r$-normal must also be $s$-normal, i.e. the property of normality, or lack thereof, remains invariant for multiplicatively dependent bases. As a result,  
\begin{equation*} 
\mathscr N(\mathscr{B}, \mathscr{B}') = \mathscr N(\overline{\mathscr{B}}, \overline{\mathscr{B}'}), \text{ where } 
\overline{\mathscr{B}} := \left\{b \in \mathbb N \setminus \{1\} \; \Bigl| \; b \sim b_0 \text{ for some } b_0 \in \mathscr{B} \right\}. 
\end{equation*}
Schmidt's result also implies that 
\[ \text{if there exist } b \in \overline{\mathscr{B}}, b' \in \overline{\mathscr{B}'} \text{ such that } b \sim b', \text{ then } \mathscr N(\mathscr{B}, \mathscr{B}') = \emptyset. \] 
Let us agree to call the pair $(\mathscr{B}, \mathscr{B}')$ {\em{compatible}} if such a pair $(b, b')$ does not exist, i.e. 
\begin{equation}  b \not\sim b' \;\; \text{ for all pairs } \; \; (b,b') \in \overline{\mathscr{B}} \times \overline{\mathscr{B}'}. \label{compatibility-definition} \end{equation} For compatible pairs $(\mathscr{B}, \mathscr{B}')$, points in $\mathscr N(\mathscr{B}, \mathscr{B}')$ are plentiful. Pollington \cite{p81} has shown that $\mathscr N(\mathscr{B}, \mathscr{B}')$ has full Hausdorff dimension for any compatible choice of $(\mathscr{B}, \mathscr{B}')$ with $\mathscr{B} \sqcup \mathscr{B}' = \mathbb N \setminus \{1\}$. In this notation, a stronger version of the question \eqref{KS-q} can be stated as: \begin{equation} \label{KS-q2}
{\text{\em{Does $\mathscr N(\mathscr{B}, \mathscr{B}')$ support a Rajchman measure for every compatible pair $(\mathscr{B}, \mathscr{B}')$?}}} 
\end{equation} 
Lyons' result \cite{l86} proves that $\mathscr{N}(\cdot, \{2\})$ supports such a measure. But even in the special case where $2 \in \mathscr{B}'$, this does not address the question posed in \eqref{KS-q2}. Indeed, a main objective of this paper is to use the Rajchman measure constructed in \cite{l86} as a case study, and explore its special features beyond the non-normality of its support in base 2. We intend to return to the question \eqref{KS-q2} in full generality in \cite{PZ-2}. 
\subsection{Statement of the main result} 
Multiplicative dependence of two integers $r$ and $s$, as defined in \eqref{def-mult-dep}, can be verified by a simple criterion: suppose that 
\[ r= \prod_{j=1}^{a} p_j^{u_j} \quad \text{ and } \quad s = \prod_{j=1}^{b} q_j^{v_j}\]
are the prime factorizations of $r$ and $s$. We will assume that  $u_j, v_j \in \mathbb N$ and that the primes $p_j, q_j$ obey the relations $p_1 < p_2 < \ldots < p_{a}$ and $q_1 < \ldots < q_b$. Then $r \sim s$ if and only if all the following three conditions hold: 
\begin{equation} \label{mult-dep-criterion} 
a = b, \quad p_j = q_j \text{ for all $j \in \{1, 2, \ldots, a\}$}, \quad \frac{u_1}{v_1} =  \frac{u_2}{v_2} = \cdots = \frac{u_a}{v_a}. 
\end{equation}   
Let $\mathscr{E}$ and $\mathscr{O}$ denote respectively the set of all even and all odd integers in $\mathbb N \setminus \{1\}$. It follows from \eqref{mult-dep-criterion} that  
\begin{itemize} 
\item $\mathtt e \not\sim \mathtt o$ for every $\mathtt e \in \mathscr{E}$, $\mathtt o \in \mathscr{O}$, 
\item if $\mathtt e \in \mathscr{E}$ and $\mathtt e \sim \mathtt e'$, then $\mathtt e' \in \mathscr{E}$, 
\item  if $\mathtt o \in \mathscr{O}$ and $\mathtt o \sim \mathtt o'$, then $\mathtt o' \in \mathscr{O}$.  
\end{itemize} 
In other words, $\overline{\mathscr{E}} = \mathscr{E}$, $\overline{\mathscr{O}} = \mathscr{O}$, $\mathscr{E} \sqcup \mathscr{O} = \mathbb N \setminus \{1\}$; i.e., $\mathscr{E}$ and $\mathscr{O}$ constitute a compatible partition of integer bases $\mathbb N \setminus \{1\}$, in the sense of \eqref{compatibility-definition}. The main result of this article is the following. 
\begin{theorem} \label{mainthm}
The set $\mathscr N(\mathscr{O}, \mathscr{E})$ of real numbers that are odd-normal but not even-normal supports a Rajchman measure
$\mu$ obeying the following property: for every $\kappa > 0$, there exists $C_{\kappa} > 0$ such that 
\begin{equation}  \bigl|\widehat{\mu}(n) \bigr| \leq \bigl(\log \log |n|  \bigr)^{-1 + \kappa} \quad \text{ for all $n \in \mathbb Z$ with } |n| \geq C_{\kappa}. \label{loglogn-bound}    
\end{equation}  
Consequently, $\mathscr N(\mathscr{O}, \mathscr{E})$ is a set of multiplicity.  
\end{theorem} 
\noindent Here $\widehat{\mu}(\cdot)$ represents the Fourier coefficient of $\mu$:
\[ \widehat{\mu}(n) := \int e(xn) d \mu(x) \; \text{ for } n \in \mathbb Z, \qquad e(y) := e^{2 \pi iy}.  \]  
The decay rate \eqref{loglogn-bound} is sharp, up to $\kappa$-loss, in the following sense.  It is known \cite[Remark 1, page 8573]{PVZZ} that if a set $E$ supports a measure $\nu$ such that for some $\kappa > 0$, 
\[ \bigl| \widehat{\nu}(n) \bigr| \leq C_{\kappa}\bigl(\log \log |n| \bigr)^{-1 - \kappa},  \]
then $\nu$-almost all numbers in $E$ are normal in every base. 
\subsection{Proof overview}
The remainder of the paper is devoted to proving Theorem \ref{mainthm}. There are three distinct steps.
\begin{itemize} 
\item The first step involves defining a class of measures $\mu = \mu[\mathcal K, \pmb{\varepsilon}]$, parametrized by a sequence of large numbers $\mathcal K$ and a sequence of small numbers $\pmb{\varepsilon}$.  This step is carried out in Section \ref{mu-construction-section}. The construction used here is the same as the one given by Lyons \cite{l86}, where a bias towards the digit 0 is built into the definition of $\mu$. A geometric description of the measure $\mu$ is given in Section \ref{mu-geometric-section}. The Rajchman property of $\mu$ follows directly from \cite{l86}.
\item The second step shows that under a mild growth condition on $\mathcal K$ and for $\mu$ as above,  $\mu$-almost every $x$ is non-normal in every even integer base $b$. This step appears in Section \ref{non-normality-section}, where Weyl's equidistribution criterion is used to disprove normality. The argument here generalizes a similar line of reasoning from \cite{l86}, which dealt with the special case $b = 2$.  
\item The third step consists of showing that $\mu$-almost every $x$ is normal in every odd base, provided the parameters $\mathcal K$ and $\pmb{\varepsilon}$ employed in the construction of $\mu = \mu[\mathcal K, \pmb{\varepsilon}]$ are appropriately chosen. The summability criterion of Davenport, Erd\H{o}s and LeVeque (Lemma \ref{DEL-lemma}) is invoked for this purpose. Verification of this criterion for the measure $\mu$ is the main new contribution of this article, and constitutes the bulk of its technical complexity. Sections \ref{normality-section-1}--\ref{special-choices-section}) contain this part of the argument, which invokes certain number-theoretic results of Schmidt \cite{s60} to establish estimates of $\widehat{\mu}$ that are finer than the ones provided by the Rajchman property. 
\end{itemize} 
A more detailed roadmap of the proof of Theorem \ref{mainthm} along with an explicit choice of $\mu$ can be found in Section \ref{proof-sketch-section}.

\subsection{Acknowledgements} 
The authors thank Dr.~Xiang Gao for introducing them to the area of metrical number theory, for valuable discussions and extensive references on the subject. This work was initiated in 2022, when JZ was visiting University of British Columbia on a study leave from China University of Mining and Technology-Beijing, funded by a scholarship from China Scholarship Council. He would like to thank all three organizations for their support that enabled his visit. JZ was also supported by National Natural Science Foundation of China  (Grant nos.~11801555, 11971058 and 12071431). MP was partially supported by a Discovery grant from Natural Sciences and Engineering Research Council of Canada (NSERC). 

\section{Construction of the measure $\mu$} \label{mu-construction-section} 
As mentioned in the introduction, we describe the construction of the measure $\mu$ in this section, in a few different ways. The parameters of the construction are the following:
\begin{itemize} 
\item A strictly increasing sequence of integers $\mathcal K = \{K_{\ell} : \ell \geq 0\}$,  with $K_0 =0$ and a corresponding partition of $\mathbb N$ into blocks 
\begin{equation} \label{block-def}  
\mathcal B_{\ell} := \{ K_{\ell-1}+1, K_{\ell-1}+2, \ldots, K_{\ell}\} \subseteq \mathbb N,
\end{equation} 
\item A sequence of real numbers $\pmb{\varepsilon} = \{\varepsilon_{\ell} : \ell \geq 1\}$, such that 
\begin{equation}  \varepsilon_{\ell} \in [0, 1] \text{ for all } \ell \geq 1, \qquad \sum_{\ell=1}^{\infty} \varepsilon_{\ell} = \infty.
\label{epsilon-assumption} \end{equation} 
\end{itemize} 
Given $\mathcal K$ and $\pmb{\varepsilon}$ as above, Lyons \cite{l86} constructed a measure $\mu$ of the following form: 
\begin{equation} \label{def-mu} 
\begin{aligned} 
\mu &= \mu[\mathcal K, \pmb{\varepsilon}] := \overset{\infty}{\underset{\ell=1}{\Asterisk}} \; \mu_{\ell}, \quad \text{ where } \\   
\mu_{\ell} := \varepsilon_{\ell} \delta(0) &+ (1 - \varepsilon_{\ell}) {\underset{k \in \mathcal B_{\ell}}{\Asterisk}} \Bigl[\frac{1}{2} \delta(0) + \frac{1}{2} \delta(2^{-k}) \Bigr] \\ 
= \varepsilon_{\ell} \delta(0) &+ (1 - \varepsilon_{\ell}) 
 {\underset{k=K_{\ell-1}+1}{\overset{K_{\ell}}{\Asterisk}}} \Bigl[\frac{1}{2} \delta(0) + \frac{1}{2} \delta(2^{-k}) \Bigr].
\end{aligned} 
\end{equation}
The infinite convolution structure \eqref{def-mu} of $\mu$ leads to the following product representation of $\widehat{\mu}$: 
\begin{align} \label{mu-hat-infinite-product} 
\widehat{\mu}(\eta) &= \prod_{\ell=1}^{\infty} \widehat{\mu}_{\ell}(\eta) = \prod_{\ell=1}^{\infty} \Bigl[ \varepsilon_{\ell} + (1 - \varepsilon_{\ell}) \mathfrak E_{\ell}(\eta) \Bigr], \; \eta \in \mathbb Z, \text{ with } \\ 
\mathfrak E_{\ell}(\eta) &:= \prod_{k = K_{\ell-1}+1}^{K_{\ell}} \Bigl[ \frac{1}{2} \bigl( 1 + e(2^{-k} \eta) \bigr) \Bigr], \quad e(y) := e^{2 \pi i y}. \label{def-El}
\end{align} 
The Rajchman property of $\mu$ was established in \cite{l86}. We recall it here:
\begin{lemma}[{\cite[Lemma 9]{l86}}] \label{Lyons-lemma}
For any choice of $\mathcal K$ and $\pmb{\varepsilon}$ as in \eqref{block-def} and \eqref{epsilon-assumption} above, let $\mu = \mu[\mathcal K, \pmb{\varepsilon}]$ be the probability measure given by \eqref{def-mu}. Then for every $n \in \mathbb Z$ with $|n| \in [2^{K_{\ell-1}-1}, 2^{K_{\ell}-1})$, the following estimate holds: 
\begin{equation} \bigl| \widehat{\mu}(n) \bigr| \leq \varepsilon_{\ell} \varepsilon_{\ell-1} + \varepsilon_{\ell} + \varepsilon_{\ell-1} + 2^{-(K_{\ell-1} - K_{\ell-2})}.  \label{mu-Rajchman} 
\end{equation} 
If $\mathcal K$ and $\pmb{\varepsilon}$ are chosen so that 
\[ K_{\ell} - K_{\ell-1} \rightarrow \infty \; \; \text{ and } \; \; \varepsilon_{\ell} \rightarrow 0, \; \; \text{ then } \; \; \widehat{\mu}(n) \rightarrow 0 \text{ as $|n| \rightarrow \infty$, i.e., $\mu$ is Rajchman.}  \]  
\end{lemma} 
\noindent The measure $\mu$ given by \eqref{def-mu} admits a few equivalent descriptions. We provide two of them in this section. 
\subsection{An alternative characterization of $\mu$} \label{mu-X-section}
Let $\{X_k : k \geq 1\}$ be a sequence of random variables in a probability space $(\Omega, \mathbb P)$  obeying the following independence and distributional properties:
\begin{itemize} 
\item Each random variable $X_k$ takes values in $\{0,1\}$. 
\item The random variables $X_k$ are block independent; specifically,
\begin{equation} \label{independence-assumption}
\bigl\{ \mathbb X_{\ell} : \ell \geq 1 \bigr\} \text{ are independent random vectors, where } \mathbb X_{\ell} := \{X_k : k \in \mathcal B_{\ell} \}. 
\end{equation} 
\item For every binary string $\pmb{\kappa} \in \{0, 1\}^{K_{\ell}-K_{\ell-1}}$, 
\begin{equation}  \label{X-dist} 
\mathbb P(\mathbb X_{\ell} = \pmb{\kappa}) = \begin{cases} (1 - \varepsilon_{\ell}) 2^{K_{\ell-1} - K_{\ell}} &\text{ if } \pmb{\kappa} \ne 0, \\ \varepsilon_{\ell} + (1 - \varepsilon_{\ell}) 2^{K_{\ell-1} - K_{\ell}}  &\text{ if } \pmb{\kappa} = 0. \end{cases} 
\end{equation}  
\end{itemize} 
Then $\mu$ given by \eqref{def-mu} is the probability distribution of the random variable \[X = \sum_{k=1}^{\infty} X_k 2^{-k}. \] More precisely, suppose that 
 \begin{equation}  x = \sum_{k=1}^{\infty} d_k(x) 2^{-k}  \in [0,1) \label{dyadic-expansion-x} \end{equation}  is a realization of $X$, and $d_k(x) \in \{0, 1\}$ denotes the $k^{\text{th}}$ digit in the binary expansion of $x$. Then $d_k(x)$ is a realization of $X_k$; further for any choice of $n \in \mathbb N$, and $(a_1, \ldots, a_n) \in \{0, 1\}^n$, 
\begin{equation} \label{mu-and-P}
\mu \bigl\{x : d_k(x) = a_k, \; 1 \leq k \leq n \bigr\} = \mathbb P(X_k = a_k, \; 1 \leq k \leq n).
\end{equation}    
\subsection{A geometric description of $\mu$ by mass distribution} \label{mu-geometric-section} 
The measure $\mu$ defined in \eqref{def-mu} can also be described through a deliberately skewed mass distribution principle. The interested reader may refer to the textbook of Falconer \cite[Chapter 1, \S 3]{Falconer-book} for a general introduction to
this principle, which involves creating a measure through repeated subdivision of a mass between parts of a bounded Borel set.  
\vskip0.1in 
\noindent Let $\mathfrak C_0$ consist of the single set $[0,1)$. In the first step, $[0,1)$ is decomposed into $2^{K_1}$ equal sub-intervals.  
\[ [0,1)  = \bigsqcup_{i_1 = 0}^{2^{K_1}-1} \mathscr{I}_{i_1}, \; \text{ where } \; \mathscr{I}_{i_1} = \bigl[ i_1 2^{-K_1}, (i_1+1) 2^{-K_1}  \bigr), \; i_1 \in \mathbb I_1 := \{0, 1, \ldots, 2^{K_1}-1\}.    \] 
We define $\mathfrak C_1 := \{\mathscr{I}_{i_1}: i_1 \in \mathbb I_1 \}$. In general, at the $\ell^{\text{th}}$ step, we have a collection $\mathfrak C_{\ell} := \{ \mathscr{I}_{\mathbf i} : \mathbf i \in \mathbb I_{\ell} \}$ of disjoint intervals of equal length $2^{-K_{\ell}}$ that partition $[0,1)$: 
\begin{align*} 
\mathscr{I}_{\mathbf i} &:= \frac{i_1}{2^{K_1}} + \frac{i_2}{2^{K_2}} + \cdots + \frac{i_{\ell}}{2^{K_{\ell}}} + \bigl[0, 2^{-K_{\ell}} \bigr), \\ 
\mathbf i &= (i_1, \ldots, i_{\ell}) \in \mathbb I_{\ell} := \prod_{k=1}^{\ell} \mathbb J_{\ell}, \quad \mathbb J_{\ell} := \bigl\{0, 1, \ldots, 2^{K_{\ell} - K_{\ell-1}}-1\bigr\}. 
\end{align*}  
At the $(\ell+1)^{\text{th}}$ step, each interval $\mathscr{I}_{\mathbf i}$ is decomposed into disjoint sub-intervals of length $2^{-K_{\ell + 1}}$. The collection $\mathfrak C_{\ell+1}$ comprises all these sub-intervals. Thus each interval $\mathscr{I}_{\mathbf i}$ in $\mathfrak C_{\ell}$ is contained in a unique interval $\mathscr{I}_{\mathbf i'} \in \mathfrak C_{\ell-1}$, called the parent of $\mathscr{I}_{\mathbf i}$. In turn, $\mathscr{I}_{\mathbf i'}$ is the parent of $2^{K_{\ell} - K_{\ell-1}}$ children of the $\ell^{\text{th}}$ generation, which are intervals of the form $\mathscr{I}_{\mathbf i} \in \mathfrak C_{\ell}$, $\mathbf i = (\mathbf i', i_{\ell})$, $i_{\ell} \in \mathbb J_{\ell}$.  Unlike a traditional Cantor-type construction that eliminates certain intervals at every stage, the union of the intervals in $\mathfrak C_{\ell}$ remains $[0,1)$ for every $\ell \geq 1$. 
\vskip0.1in
\noindent In contrast, the mass distribution scheme $\mu$ treats the different basic intervals differently at any given stage. We endow unit mass to the interval $[0,1)$ in $\mathfrak C_0$. This mass is split between the intervals in $\mathfrak C_1$ in the folowing way: first, a mass of $(1 - \varepsilon_1)$ is distributed equally among all the intervals $\mathscr{I}_{i_1}$; the remaining mass $\varepsilon_1$ is then assigned entirely to $\mathscr{I}_0$, giving it preferential weightage. Thus, for $i_1 \in \mathbb I_1$, 
\[ \mathtt m_{i_1} = \mu(\mathscr{I}_{i_1}) = \begin{cases} (1 - \varepsilon_1) 2^{-K_1} &\text{ if } i_1 \in \mathbb I_1 \setminus \{0\}, \\ \varepsilon_1 +
 (1 - \varepsilon_1) 2^{-K_1} &\text{ if } i_1 = 0. \end{cases} \] 
This skewed mass distribution mechanism favouring blocks of zeroes persists at each step. For every $\mathbf i' \in \mathbb I_{\ell-1}$, 
\begin{align*}  
\mathtt m_{\mathbf i'} = \mu(\mathscr{I}_{\mathbf i'}) &= \sum_{i_{\ell} \in \mathbb J_{\ell}} \mu(\mathscr{I}_{\mathbf i}) = \sum_{i_{\ell} \in \mathbb J_{\ell}} \mathtt m_{\mathbf i} \quad \text{ with } \mathbf i = (\mathbf i', i_{\ell}), \\ \frac{\mathtt m_{\mathbf i}}{\mathtt m_{\mathbf i'}} = \frac{\mu(\mathscr{I}_{\mathbf i})}{\mu(\mathscr{I}_{\mathbf i'})} &= 
\begin{cases} (1 - \varepsilon_\ell) 2^{-(K_\ell - K_{\ell-1})} &\text{ if } i_\ell \in \mathbb J_\ell \setminus \{0\}, \\ \varepsilon_\ell +
 (1 - \varepsilon_\ell) 2^{-K_\ell - K_{\ell-1}} &\text{ if } i_\ell = 0. \end{cases} 
  \end{align*} 
The above description of $\mu$ extends to a unique Borel probability measure  \cite[Proposition 1.7]{Falconer-book}, which coincides with the alternative definitions of $\mu$ given in \eqref{def-mu} and in Section \ref{mu-X-section}.  Let us also note in passing that $\mu$ is the weak-$\ast$ limit of the absolute continuous probability measures $\nu_{\ell}$ on $[0,1)$ given by $\nu_{\ell} := 2^{K_{\ell}} \sum_{\mathbf i \in \mathbb I_{\ell}} \mathtt m_{\mathbf i} \mathbf 1_{\mathscr{I}_{\mathbf i}}$.  
\subsection{Layout of proof of Theorem \ref{mainthm}} \label{proof-sketch-section} 
We will present the main components of the proof here, and point the reader to later sections of the paper where these aspects are treated in detail.  Let us choose the following parameters $\mathcal K$ and $\pmb{\varepsilon}$ for constructing $\mu = \mu[\mathcal K, \pmb{\varepsilon}]$ as in \eqref{def-mu}: for a large integer $K \geq 10$, 
\begin{equation} \label{choice-of-K-epsilon} 
K_{\ell} = K^{\ell \omega(\ell)} \text{ with } \omega(\ell) := \lfloor \sqrt{\log \ell} \rfloor,  \qquad \varepsilon_{\ell} = \ell^{-1}.   
\end{equation}  
Then Lemma \ref{Lyons-lemma} implies that $\mu$ is Rajchman. Indeed for the choice \eqref{choice-of-K-epsilon} of $\mathcal K$ and $\pmb{\varepsilon}$, the inequality \eqref{mu-Rajchman} yields an explicit pointwise decay estimate of $\mu$:  
\begin{align*} 
\bigl| \widehat{\mu}(n) \bigr| &\leq \frac{4}{\ell-1} \quad \text{ for } 2^{K_{\ell-1} -1} \leq |n| < 2^{K_{\ell}-1} \\ 
&\leq (\log \log |n|)^{-1 + \kappa} \text{ for all } \kappa > 0 \text{ and } |n| \geq C_{K, \kappa}. 
\end{align*}
The last step follows from the assumption $n < 2^{K_{\ell}-1}$, which is equivalent to the estimate $\ell \omega(\ell) \log_2 K \geq \log_2(\log_2 |n| +1)$. This establishes \eqref{loglogn-bound}. 
\vskip0.1in 
\noindent It remains to show that $\mu$-almost every point is odd-normal but not even-normal. This is executed in the subsequent sections. The statement concerning failure of normality in even bases appears in Proposition \ref{prop2} and is proved in Section \ref{non-normality-section}. The proof of normality in odd bases is substantially more complex. A well-known criterion for checking normality, due to Davenport, Erd\H{o}s and LeVeque \cite{del63}, is recalled in Section \ref{normality-section-1}. Proposition \ref{mu-odd-normal} posits that this criterion is satisfied by $\mu$. Verification of this proposition takes up the remainder of the paper. 
\qed

\section{Non-normality in even bases}    \label{non-normality-section}
\subsection{Normality and distribution modulo 1} \label{unif-dist-section}
Several equivalent formulations of normality are ubiquitous in the literature. One of the earliest characterizations involves the notion of uniformly distributed sequences. We briefly recall the definition:  a sequence $(x_n)_{n \geq 1}$ is said to be {\em{uniformly distributed modulo 1}} if for every $u, v \in [0, 1]$, $u < v$, the following condition holds:
\[ \lim_{N \rightarrow \infty} \frac{1}{N} \# \Bigl\{1 \leq n \leq N: \{x_n \} \in [u, v) \Bigr\} = v - u. \]
A classical theorem known as Weyl's criterion \cite[Theorem 1.2]{b12} states that $(x_n)_{n \geq 1}$ is uniformly distributed modulo 1 if and only if 
\begin{equation} \label{Weyl}
\lim_{N \rightarrow \infty} \frac{1}{N} \sum_{n=1}^{N} e(h x_n)  = 0 \quad \text{ for all non-zero integers } h. 
\end{equation}   
The following result, proved by Wall \cite{w49} in his Ph.D. thesis, connects normality with uniform distribution mod one; a proof can also be found in \cite[Theorem 4.14]{b12}. 
\begin{lemma}[\cite{w49}, {\cite[Theorem 4.14]{b12}}] \label{Weyl-lemma} 
Let $b \geq 2$. Then a real number $\xi$ is normal to base $b$ if and only if $(\xi b^n)_{n \geq 1}$ is uniformly distributed modulo 1. In view of \eqref{Weyl}, this is equivalent to the condition:
\begin{equation} \label{Weyl-2} 
\lim_{N \rightarrow \infty} \frac{1}{N} \sum_{n=1}^{N} e(h \xi b^n)  = 0 \quad \text{ for all non-zero integers } h. 
\end{equation}  
\end{lemma} 
\subsection{Simultaneous non-normality}
The main result in this section is the following. 
\begin{proposition} \label{prop2}
Let $\mathcal K = \{K_{\ell}\} \subseteq \mathbb N$ be any strictly increasing sequence such that 
\begin{equation} \label{K-ratio}
\limsup_{\ell \rightarrow \infty} \frac{K_{\ell}}{K_{\ell+1}} = 0.   
\end{equation} 
For any sequence $\pmb{\varepsilon}$ obeying \eqref{epsilon-assumption}, let $\mu$ be the probability measure defined in \eqref{def-mu}. Then the condition \eqref{Weyl-2} fails with $h = 1$ for every even integer $b \in \mathbb N$ and for $\mu$-almost every $x \in [0,1)$. Specifically, there exists a Borel set $\mathfrak A$ with $\mu(\mathfrak A) = 1$ such that 
\begin{equation} \label{non-normal-condition}
 \limsup_{N \rightarrow \infty} \frac{1}{N} \sum_{k=1}^{N} \text{Re}(e(xb^n)) = 1  \quad \text{  for all $x \in \mathfrak A$ and } b \in 2\mathbb N.   
 \end{equation} 
Consequently, $\mu$ is supported on $\mathscr N(\cdot, \mathscr{E}) \cap [0,1]$. Here $\mathscr{E}$ denotes the collection of all even integers and $\mathscr N(\cdot, \mathscr E)$ defined as in \eqref{N2} is the set of real numbers that are not normal with respect to any even base.  
\end{proposition} 
\noindent It is straightforward to verify that \eqref{K-ratio} holds for the choice of integers $\mathcal K = \{K_{\ell}: \ell \geq 1\}$ given by \eqref{choice-of-K-epsilon}. Thus the conclusion of Proposition \ref{prop2} holds for the measure $\mu$ defined in the proof of Theorem \ref{mainthm} in Section \ref{proof-sketch-section}.   
\subsection{Proof of Proposition \ref{prop2}}
\subsubsection{Definition of the set $\mathfrak A$} 
The statement of Proposition \ref{prop2} posits a set $\mathfrak A$ of full $\mu$-measure, every point of which is non-normal in even bases. We describe the set $\mathfrak A$ in this section.
\vskip0.1in 
\noindent For $x \in [0,1)$, let us write $x$ as \eqref{dyadic-expansion-x}, 
and set 
\begin{equation} 
\mathfrak A_{\ell} := \bigl\{x \in [0,1) : d_k(x) = 0 \text{ for all } k \in \mathcal B_{\ell}\bigr\},
\end{equation} 
where $\mathcal B_{\ell}$ is the $\ell^{\text{th}}$ block of integers defined in \eqref{block-def}. 
In other words for every $x \in \mathfrak A_{\ell}$, the binary string of digits corresponding to the base 2 expansion of $x$ contains a block of zeros of length at least $K_{\ell} - K_{\ell-1}$ starting from the $(K_{\ell-1}+1)^{\text{th}}$ entry. 
It follows from \eqref{mu-and-P} and \eqref{X-dist} that 
\begin{equation}
\mu(\mathfrak A_{\ell}) = \mathbb P(\mathbb X_{\ell} = 0) = \varepsilon_{\ell} + (1 - \varepsilon_{\ell}) 2^{K_{\ell-1}-K_{\ell}} \geq \varepsilon_{\ell}.
\end{equation}  
Our choice \eqref{epsilon-assumption} of $\pmb{\varepsilon}$ leads to the observation 
\begin{equation} \label{divergent-sum} 
\sum_{\ell} \mu(\mathfrak A_{\ell}) = \infty.
\end{equation} 
The set $\mathfrak A$ whose existence has been claimed in Proposition \ref{prop2} is now defined as follows: 
\begin{align} 
\mathfrak A & := \limsup_{\ell \rightarrow \infty} \mathfrak A_{\ell} = \bigcap_{k=1}^{\infty} \bigcup_{\ell = k}^{\infty} \mathfrak A_{\ell} = \bigl\{x \in [0,1) : x \in \mathfrak A_{\ell} \text{ for infinitely many indices } \ell \bigr\} \nonumber \\ 
&= \bigl\{x \in [0,1] : d_k(x) = 0 \text{ for all $k \in \mathcal B_{\ell}$ for infinitely many indices } \ell \bigr\}.  \label{def-A}
\end{align}  
The discussion in Section \ref{mu-X-section} establishes that the event $\mathfrak A_{\ell}$ depends only on the random vector $\mathbb X_{\ell}$. The collection of events $\{\mathfrak A_{\ell} : \ell \geq 1 \}$ is therefore independent, by the assumption \eqref{independence-assumption} of mutual independence of the random vectors $\{\mathbb X_{\ell} : \ell \geq 1\}$. Combined with the divergence condition \eqref{divergent-sum}, this permits the application of the Borel-Cantelli lemma \cite[Theorem 2.3.6]{d10}, which yields $\mu(\mathfrak A) = 1$. 

\subsubsection{Proof of \eqref{non-normal-condition}} 
Choose any $b \in 2\mathbb N$, and any $x \in \mathfrak A$. In order to prove \eqref{non-normal-condition}, it suffices to find a subsequence of integers $N_j \rightarrow \infty$ (possibly depending on $b$ and $x$) such that
\begin{equation} \label{limsup-1} 
 \lim_{j \rightarrow \infty}\frac{1}{N_j} \sum_{k=1}^{N_j} \text{Re}(e(xb^n)) \geq 1. 
\end{equation} 
Since each summand is at most 1, the above relation would establish that
\[ 1 \leq \limsup_{N \rightarrow \infty}\frac{1}{N} \sum_{k=1}^{N} \text{Re}(e(xb^n)) \leq 1, \text{ which is \eqref{non-normal-condition}}.\] 
For every $x \in \mathfrak A$, the defining condition \eqref{def-A} ensures the existence of a sequence of integers $\ell_j \nearrow \infty$ depending on $x$ such that 
\begin{equation}  \label{zero-blocks}
d_k(x) = 0 \quad \text{ for all } k \in \mathcal B_{\ell}, \; \ell \in \{\ell_j : j \geq 1 \}. 
\end{equation}  
The sequence $N_j$ realizing \eqref{limsup-1} will be chosen depending on $\ell_j$. 
\vskip0.15in
\noindent We set about proving \eqref{limsup-1}. For any base $b \in 2 \mathbb N$, let us write 
\begin{equation} \label{even-b} 
b = 2^{\mathtt M} \mathtt B \quad \text{ where } \quad \mathtt M, \mathtt B \in \mathbb N, \; 2 \nmid \mathtt B. 
\end{equation} 
This means that $2^{\mathtt M} \leq b$, and hence $\mathtt M/\log_2 b \in (0,1]$. Let us fix a constant $\alpha$:
\begin{equation} \label{def-alpha} 
0 < \alpha < \frac{\mathtt M}{\log_2 b} \leq 1. 
\end{equation}    
For any sufficiently large index $j$ depending on $b$, let us set $\ell = \ell_j$, and define 
\begin{align} 
&N = N_j  := 1 + \Bigl\lfloor \frac{\alpha}{\mathtt M} K_{\ell} \Bigr\rfloor, \quad N' = N'_j := 1 + \Bigl\lfloor \frac{1}{\mathtt M} K_{\ell-1} \Bigr\rfloor ,  \text{ so that } \label{N-N'} \\  &\frac{N'}{N} = \frac{N'_j}{N_j}\leq 2{\Bigl\lfloor \frac{1}{\mathtt M} K_{\ell-1} \Bigr \rfloor}/{\Bigl\lfloor \frac{\alpha}{\mathtt M} K_{\ell} \Bigr\rfloor} \rightarrow 0 \text{ as } \ell = \ell_j \rightarrow \infty. \label{N'/N}
\end{align} 
The floor function $\lfloor a \rfloor$ in the display above denotes the integral part of $a$. The limit in \eqref{N'/N} follows from the hypothesis \eqref{K-ratio}. We also set 
\begin{equation}   
\mathbb N_1 = \mathbb N_1(N) := \bigl\{1, 2, \ldots, N' \bigr\}, \quad \mathbb N_2 = \mathbb N_2(N) := \{N' + 1, \ldots, N\}.   
\end{equation} 
The above splitting of $\{1, 2, \ldots, N\}$ results in a corresponding 
decomposition of the exponential sum $\mathcal S$: 
\begin{align} \label{S-formula}
\mathcal S &= \mathcal S(x) := \sum_{n=1}^{N} e(xb^n) = \mathcal S_1(x) + \mathcal S_2(x), \quad \text{ where } \\ \mathcal S_i 
&= \mathcal S_i(x) = \sum_{n \in \mathbb N_i} e(xb^n), \; i=1,2. 
\end{align}   
The main claim concerning $\mathcal S_2$ is the following: for $x \in \mathfrak A$, 
\begin{equation} \label{S2-claim}
\text{Re}(\mathcal S_2(x)) \geq (N-N') \cos \bigl(2 \pi b c^{K_{\ell}} \bigr) \text{ where } c = \frac{1}{2} b^{\frac{\alpha}{\mathtt M}}.
\end{equation} 
We will prove \eqref{S2-claim} momentarily in Section \ref{S2-claim-section} below. Assuming this, the proof of \eqref{non-normal-condition} is completed as follows. The choice  \eqref{def-alpha} of $\alpha$ implies that $c \in (0,1)$. As a result, 
\begin{equation}  \label{cosine-lower-bound}   
c^{K_{\ell}} \rightarrow 0 \text{ as } \ell \rightarrow \infty, \quad \text{ and hence } \quad \cos(2 \pi b c^{K_{\ell}}) \rightarrow 1. 
\end{equation}  
Inserting \eqref{S2-claim} into the expression \eqref{S-formula} for $\mathcal S$, and using the trivial estimate $|\mathcal S_1| \leq N'$, we arrive at the relation
\begin{align*}
\text{Re}(\mathcal S(x)) = \text{Re}(\mathcal S_2(x)) + \text{Re}(\mathcal S_1(x)) \geq (N-N') \cos \bigl(2 \pi b c^{K_{\ell}} \bigr)  -  N'
\end{align*} 
for every $x \in \mathfrak A$. Consequently, 
\begin{align*} 
1 &\geq \frac{1}{N} \sum_{k=1}^{N} \text{Re}(e(xb^n)) = \frac{\text{Re}(\mathcal S(x))}{N} \\
&\geq \left(1 - \frac{N'}{N}\right) \cos \bigl(2 \pi b c^{K_{\ell}} \bigr)  - \frac{N'}{N} \rightarrow 1 \text{ as } N = N_j \rightarrow \infty.
\end{align*} 
The evaluation of the limit in the last step follows from \eqref{N'/N} and \eqref{cosine-lower-bound}. This establishes \eqref{limsup-1}.  
\qed
\subsubsection{Proof of \eqref{S2-claim}} \label{S2-claim-section}  
It remains to prove the lower bound on \eqref{S2-claim} on the real part of the exponential sum $\mathcal S_2$. Let us observe that
\[ \text{Re}(\mathcal S_2) = \sum_{n=N'+1}^{N} \cos(2 \pi xb^n) = \sum_{n=N'+1}^{N} \cos(2 \pi \{ xb^n \}) \]
involves $(N-N')$ summands, and $\cos( \cdot)$ is a decreasing function in the domain $[0, \frac{\pi}{2}]$. Thus the inequality \eqref{S2-claim} will follow if we show that for $x \in \mathfrak A$, 
\begin{equation} 
\{xb^n\} \leq b c^{K_{\ell}} < \frac{1}{4}, \quad \text{ for all } n \in \mathbb N_2 = \{N'+1, \ldots, N\}.  
\end{equation} 
Here $\{a\} = a - \lfloor a \rfloor$ denotes the fractional part of $a$. 
\vskip0.1in 
\noindent We estimate $\{xb^n\}$ using \eqref{dyadic-expansion-x} and \eqref{even-b}:
\[ xb^n= \sum_{k=1}^{\infty} (2^{\mathtt M} \mathtt B)^n d_k(x) 2^{-k}, \quad \text{ which implies } \quad  \{xb^n\}= \Bigl\{ \mathtt B^n \sum_{k=\mathtt Mn+1}^{\infty} d_k(x) 2^{-k + \mathtt Mn} \Bigr\}.  \] 
Even though the first index in the sum representing $\{xb^n\}$ above appears to be $k = \mathtt Mn+1$, we will now show that the first non-zero entry occurs much later. The definition of $N, N'$ in \eqref{N-N'} and the range of $n \in \mathbb N_2$ gives 
\begin{align*}  
&\frac{1}{\mathtt M} K_{\ell-1} < N'  < n \leq N \leq 1 + \frac{\alpha}{\mathtt M} K_{\ell}, \text{ which in turn shows that } \\ 
&\mathtt Mn + 1 \in [K_{\ell-1} + 1, \alpha K_{\ell} + 2\mathtt M] \cap \mathbb N \subseteq \mathcal B_{\ell} = [K_{\ell-1}+1, K_{\ell}] \cap \mathbb N.  
\end{align*} 
The last containment holds for $\ell$ sufficiently large, so that $\alpha K_{\ell} + 2 \mathtt M < K_{\ell}$. The defining condition \eqref{zero-blocks} for $x \in \mathfrak A$ says that $d_k(x) = 0$ for all $k \in \mathcal B_{\ell}$; this results in the following reduction: for $n \in \mathbb N_2 = \{N' + 1, \ldots, N\}$, 
\begin{align*} 
\{xb^n\} &= \mathtt B^n \sum_{k = K_{\ell} + 1}^{\infty} d_k(x) 2^{-k + \mathtt Mn} = b^n \sum_{k = K_{\ell} + 1}^{\infty} d_k(x) 2^{-k} \\ 
&\leq b^{N} \sum_{k = K_{\ell} + 1}^{\infty} 2^{-k} = b^N 2^{-K_{\ell}} \\ 
&\leq b^{1 + \frac{\alpha K_{\ell}}{\mathtt M}} 2^{-K_{\ell}}  = b c^{K_{\ell}}, 
\end{align*}  
with $c$ as in \eqref{S2-claim}. The penultimate step uses the fact that $N \leq 1 + \alpha K_{\ell}/\mathtt M$, which is a consequence of the definition \eqref{N-N'}. This completes the proof of Proposition \ref{prop2}.

\section{Normality in odd bases} \label{normality-section-1}
Next we turn our attention to the following question: are points in supp$(\mu)$ normal in certain bases? Clearly even integer bases are eliminated from this consideration, by virtue of the analysis in Section \ref{non-normality-section}. Our goal is to show that the answer is affirmative for all other integer bases. Precisely, $\mu$-almost every point is normal in every odd base, provided the parameters $\mathcal K$ and $\pmb{\varepsilon}$ in the construction of $\mu = \mu[\mathcal K, \pmb{\varepsilon}]$ are suitably chosen. This section outlines the broad strokes of this argument. Details of the proof are relegated to later sections. 
\subsection{A sufficient condition for proving $b$-normality}
Let us recall the discussion in Section \ref{unif-dist-section}, and specifically Lemma \ref{Weyl-lemma}, which relates normality with uniform distribution modulo 1.  A fundamental tool in the theory of uniform distribution is the theorem of Davenport, Erd\H{o}s and LeVeque \cite{del63}. It relates generic distribution properties of a sequence of the form $\{s_n(x) : n \geq 1\}$ with $x \in \text{supp}(\mu) \subseteq [0,1]$ with Fourier decay properties of $\widehat{\mu}$. In view of Lemma \ref{Weyl-lemma} and specializing to $s_n(x) = b^n x$, this result also provides a sufficient condition for establishing generic $b$-normality of numbers on supp$(\mu)$. While the criterion was originally stated in \cite{del63} for the Lebesgue measure on $[0,1]$, the same proof generalizes verbatim for any probability measure $\mu$, and is now ubiquitous in that generalized form in the literature. We state the latter version with appropriate references.       
\begin{lemma}[\cite{del63}, {\cite[Theorem 1.3]{Gao-Ma-Song-Zhang}}, {\cite[Theorem DEL]{PVZZ}}] \label{DEL-lemma}
Let $\mu$ be a probability measure on $[0,1]$ and $\{s_n(\cdot) : n \geq 1\}$ a sequence of measurable functions, also on $[0,1]$. If 
\begin{equation} \label{DEL-v1} 
\sum_{N=1}^{\infty} N^{-1} \int_{0}^{1} \Bigl| \frac{1}{N} \sum_{n=1}^{N} e \bigl(h s_n(x) \bigr)\Bigr|^2 \, d\mu(x) < \infty \quad \text{ for every } h \in \mathbb Z \setminus \{0\},  
\end{equation}  
then $\{s_n(x) : n \geq 1\}$ is uniformly distributed modulo 1 for $\mu$-almost every $x$. 
\vskip0.1in 
\noindent Suppose $b \in \mathbb N \setminus \{1\}$. Specializing \eqref{DEL-v1} to the case $s_n(x) = b^n x$ and applying Lemma \ref{Weyl} leads to the following statement: if  for every $h \in \mathbb Z \setminus \{0\}$, 
\begin{align} 
&\sum_{N=1}^{\infty} N^{-3} \sum_{u = 1}^{N} \sum_{v = 1}^{N} \widehat{\mu}(h(b^v-b^u)) < \infty, \; \text{ or equivalently } \nonumber \\ 
&\sum_{N=1}^{\infty} N^{-3} \sum_{u = 1}^{N} \sum_{v = 1}^{N-u} \widehat{\mu}(hb^u (b^v-1)) < \infty, 
\end{align}  
then $\mu$-almost every $x$ is $b$-normal.  
\end{lemma} 
\noindent We will use Lemma \ref{DEL-v1} to prove that $\mu$-almost every $x$ is $r$-normal where $r \geq 3$ is odd and $\mu$ is as in \eqref{def-mu}.   
The main result of this section is the following. 
\begin{proposition} \label{mu-odd-normal}
There exists a choice of $\mathcal K$ and $\pmb{\varepsilon}$ as in \eqref{block-def} and \eqref{epsilon-assumption}, for instance \eqref{choice-of-K-epsilon}, with the following property. For $\mu = \mu[\mathcal K, \pmb{\varepsilon}]$ as in \eqref{def-mu} and any odd integer $r \in \mathbb N \setminus \{1\}$, the following relation holds:  
\begin{equation} \label{I-est}
\sum_{N=1}^{\infty} N^{-3} \mathtt I(h; r, N) < \infty \quad \text{  where } \quad 
\mathtt I(h;r, N) := \sum_{v = 1}^{N} \sum_{u = 1}^{N} \bigl| \widehat{\mu}(hr^u(r^v-1)) \bigr|. 
\end{equation} 
As a result, by the Davenport-Erd\H{o}s-LeVeque criterion (Lemma \ref{DEL-lemma}) and for this choice of $\mathcal K, \pmb{\varepsilon}$,  
\[\text{$\mu = \mu[\mathcal K, \pmb{\varepsilon}]$-almost every $x$ is $r$-normal for every odd integer $r \geq 3$.} \] More precisely, $\mu$ is supported on $\mathscr N(\mathcal O, \cdot) \cap [0,1)$, the set of real numbers in $[0, 1)$ that are normal to every odd base. 
\end{proposition}
\noindent The proof structure of Proposition \ref{mu-odd-normal} is outlined in Figure \ref{layout-fig}. In Section \ref{I-est-section}, we will decompose the sum in \eqref{I-est} into two sub-sums, which will be treated separately in Propositions \ref{Ii-lemma-parta} and \ref{Ii-lemma-partb}. The first of these (Proposition \ref{Ii-lemma-parta}) is quite general, in the sense that it holds for any choice of parameters $\mathcal K, \pmb{\varepsilon}$ for which $\mu$ as in \eqref{def-mu} can be defined. In contrast, Proposition \ref{Ii-lemma-partb} places stronger restrictions on 
$\mathcal K, \pmb{\varepsilon}$. The necessary criterion is termed ``admissibility with exponent $\gamma > 1$''. In Lemma \ref{special-choices-lemma}, we verify that the choice \eqref{choice-of-K-epsilon} of parameters is admissible with a valid exponent $\gamma$, so that both propositions apply.      
\begin{center}
\begin{figure}
\begin{tikzpicture}[node distance=2cm]
\node (thm) [box, align=center] {Normality in odd bases \\ Proposition \ref{mu-odd-normal}};
\node (S3) [box, below of=thm, xshift=-5cm, yshift=-0.5cm, align=center]
{Proposition \ref{Ii-lemma-parta}; \\  Holds for any $\mathcal K, \pmb{\varepsilon}$ \\ obeying \eqref{block-def} and \eqref{epsilon-assumption}};
\node (S4) [box, below of=thm, xshift=5cm, yshift=-0.5cm, align=center] {Proposition \ref{Ii-lemma-partb}; needs $\mathcal K, \pmb{\varepsilon}$ with \\ \eqref{block-def} and \eqref{epsilon-assumption}, also admissible \\ with $\gamma > 1$, see page \pageref{tT-condition}};
\node (S2) [box, below of=thm, xshift=0cm, yshift=-3cm,align=center] {$\mathcal K, \pmb{\varepsilon}$ as in \eqref{choice-of-K-epsilon}};.

\draw [arrow] (S3) -- (thm) ;
\draw [arrow] (S4) -- (thm) ;
\draw [arrow] (S2) -- (S3) ;
\draw [arrow] (S2) -- (S4) ;
\end{tikzpicture}
\caption{Proof scheme of Proposition \ref{mu-odd-normal}} \label{layout-fig}
\end{figure}
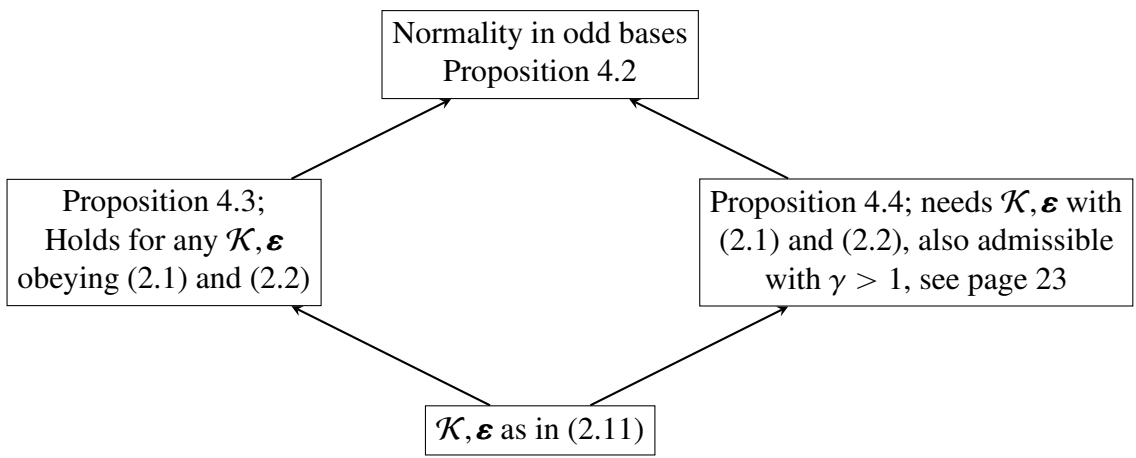 
\end{center} 
\subsection{Estimation of $\mathtt I$} \label{I-est-section}
The pointwise estimate \eqref{mu-Rajchman} on $\widehat{\mu}$, while sufficient for ensuring the Rajchman property of $\mu$, is too weak to ensure summability  of $N^{-3} \mathtt I(h; r, N)$. Estimation of the double sum \eqref{I-est} representing $\mathtt I$, ultimately leading to the convergence of the series $\sum_N \mathtt I(\cdot ; \cdot, N)/N^3$, rests on the following key steps: 
\begin{itemize} 
\item If the integer $n=hr^u(r^v-1)$ possesses certain special arithmetic structures, then the Fourier coefficient $\widehat{\mu}(n)$ is large, close to the upper bound presented in \eqref{mu-Rajchman}. Such a situation could arise, for instance, when $n$ is divisible by a large power of 2, or the binary expansion of $n$ has many long unbroken runs of the digits 0 or 1. Fortunately, the number of indices $(u,v)$ where this happens is small, thanks to the assumption that $r$ is odd. Precise statements in support of this phenomenon appear in Lemmas \ref{I1-lemma}, \ref{digit-lemma-2} and \ref{U1vtau-cardinality-lemma}.
\vskip0.1in
\item For the vast majority of choices of $(u, v) \in \{1, \ldots, N\}^2$, the Fourier coefficient $\widehat{\mu}(hr^u(r^v-1))$ is much smaller than its maximum possible value given by \eqref{mu-Rajchman}, in fact, small enough to meet the summability requirement of the Davenport-Erd\H{o}s-LeVeque criterion. The relevant estimate is captured in Lemma \ref{J1-lemma}.         
\end{itemize} 
Let us make the above statements precise. For the remainder of the article, we will fix $N \in \mathbb N$ and define an integer $R$ that is related to $N$ by the inequality  
\begin{equation} \label{N-and-R}
2^{R-1} < N \leq 2^R \quad \text{ for some } R \in \mathbb N.   
\end{equation} 
Given any $r \geq 3$, we decompose $\{1, 2, \ldots, N\}$ as follows: 
\begin{align}
\mathbb V_1 = \mathbb V_1(r, N) &:= \bigl\{1 \leq v \leq N : 2^{R_0} \, | \, r^{v}-1 \bigr\}, \text{ where } R_0 = \lfloor \sqrt{R} \rfloor, \label{def-W1} \\ 
\mathbb V_2 = \mathbb V_2(r, N) &:= \bigl\{1, \ldots, N\} \setminus \mathbb V_1(r, N).  \label{def-W2}  
\end{align}
This results in a corresponding decomposition of the quantity $\mathtt I(h;r, N)$ defined in \eqref{I-est}: 
\begin{align} 
\mathtt I(h; r, N) &= \mathtt I_1(h; r, N) + \mathtt I_2(h;r, N), \text{ where } \label{I-decomp} \\ 
\mathtt I_i(h;r, N) &:=  \sum_{v \in \mathbb V_i} \sum_{u=1}^{N}\bigl| \widehat{\mu}(hr^u(r^v-1)) \bigr|, \quad i=1,2. \label{def-Ii}
\end{align} 
The crucial estimates concerning $\mathtt I_i(h;r, N)$ are summarized in the two propositions below. 
\begin{proposition} \label{Ii-lemma-parta}
For any choice of parameters $\mathcal K$ and $\pmb{\varepsilon}$ obeying \eqref{block-def} and \eqref{epsilon-assumption}, the probability measure $\mu = \mu[\mathcal K, \pmb{\varepsilon}]$ given by \eqref{def-mu} obeys the following property. 
\vskip0.1in 
\noindent Given any odd integer $r \geq 3$ and $h \in \mathbb Z \setminus \{0\}$, there exist positive constants $c_0, C_0$ depending only on $r$ and $h$ such that for any $N \in \mathbb N$ in the range \eqref{N-and-R}, the sum $\mathtt I_1(h;r, N)$ given by \eqref{def-Ii} obeys the following estimate:  
\begin{equation} 
\mathtt I_1(h;r, N) \leq C_0 N^{2 - \frac{c_0}{\sqrt{\log_2 N}}} \leq C_0 2^{2R - c_0\sqrt{R}}. \label{I1-est}
\end{equation} 
\end{proposition} 
\begin{proposition} \label{Ii-lemma-partb}
There exists a choice of parameters $\mathcal K$ and $\pmb{\varepsilon}$ obeying \eqref{block-def} and \eqref{epsilon-assumption}, specifically \eqref{choice-of-K-epsilon}, for which the measure $\mu = \mu[\mathcal K, \pmb{\varepsilon}]$ given by \eqref{def-mu} obeys the following property.  
\vskip0.1in
\noindent For any odd integer $r \geq 3$ and  $h \in \mathbb Z \setminus \{0\}$, one can find constants $c_0, C_0 > 0, \gamma > 1$ depending only on the above quantities such that for any $N \in \mathbb N$ in the range \eqref{N-and-R}, the sum $\mathtt I_2(h;r, N)$ given by \eqref{def-Ii} obeys the following estimate:  
\begin{align}
\mathtt I_2(h;r, N) &\leq C_0 \bigl[  N^{2- \frac{c_0}{\sqrt{\log_2 N}}} + N^{2} (\log_2 N)^{-\gamma} \bigr] \nonumber \\ 
&\leq C_0 2^{2R} \bigl[2^{- c_0 \sqrt{R}} + R^{-\gamma} \bigr]. \label{I2-est} 
\end{align}  
\end{proposition} 
\begin{lemma} \label{special-choices-lemma}
The conclusion of Proposition \ref{Ii-lemma-partb} holds for the choice \eqref{choice-of-K-epsilon} of $\mathcal K$ and $\pmb{\varepsilon}$.  
\end{lemma} 
\noindent The proofs of Propositions \ref{Ii-lemma-parta} and \ref{Ii-lemma-partb} will be presented in Sections \ref{I1-lemma-proof-section} and \ref{I2-lemma-proof-section} respectively. Lemma \ref{special-choices-lemma} is verified in Section \ref{special-choices-section}. Assuming the estimates \eqref{I1-est} and \eqref{I2-est} for now, the proof of Proposition \ref{mu-odd-normal} is completed as follows. 
\subsection{Proof of Proposition \ref{mu-odd-normal}} 
Let us choose $\mathcal K$ and $\pmb{\varepsilon}$ as in \eqref{choice-of-K-epsilon}. Lemma \ref{special-choices-lemma} then asserts the validity of the conclusions of both Propositions \ref{Ii-lemma-parta} and \ref{Ii-lemma-partb}. We will establish the convergence of the sum \eqref{def-Ii} for this choice of $\mathcal K$ and $\pmb{\varepsilon}$.  
\vskip0.1in 
\noindent Since the summands of $\mathtt I$ are all positive, $\mathtt I(h;r, N)$ is monotone non-decreasing in $N$. We decompose the sum in \eqref{I-est} in dyadic blocks, using the decomposition \eqref{I-decomp} and the estimates \eqref{I1-est}, \eqref{I2-est} to bound the contribution from each block. This yields  
\begin{align*}
 \sum_{N=1}^{\infty} N^{-3} \mathtt I(h;r, N) &=  \sum_{R=1}^{\infty} \sum_{N = 2^{R-1}}^{2^R - 1} N^{-3}  \mathtt I(h;r, N) \\ 
 &\leq \sum_{R=1}^{\infty} (2^{R-1})^{-3} 2^{R-1} \mathtt I(h; r, 2^R) \\ 
 & = 2^2 \sum_{R=1}^{\infty} 2^{-2R} \bigl[ \mathtt I_1(h; r, 2^R) + \mathtt I_2(h; r, 2^R) \bigr] \\ 
 &\leq C_0 \sum_{R=1}^{\infty} \bigl[ 2^{-c_0\sqrt{R}} + R^{-\gamma} \bigr] < \infty, 
 \end{align*} 
 where the last step follows from the assertions that $c_0 > 0$ and $\gamma > 1$. This completes the proof of Proposition \ref{mu-odd-normal}.
\qed

\section{Estimation of $\mathtt I_1$} \label{I1-lemma-proof-section} 
The goal of this section is to prove Proposition \ref{Ii-lemma-parta}. No special properties of $\mu$ will be employed. The proof follows from the fact that members of $\mathbb V_1$ are relatively sparse in $\{1, 2, \cdots, N\}$. This is made precise in Lemma \ref{I1-lemma}. Some necessary tools are collected en route in Section \ref{integer-order-section}   
\subsection{Order of an integer} \label{integer-order-section} 
For any two coprime integers $a$ and $m$, the {\em{order}} of $a$ (mod $m$), denoted $\text{ord}_m(a)$ is given by
\begin{equation} 
\text{ord}_m(a) := \min\bigl\{n \in \mathbb N: a^n \equiv 1 \; (\text{mod } m) \bigr\}. 
\end{equation}  
Let us record two facts concerning $\text{ord}_m(a)$ which will be useful in the sequel. 
\begin{lemma}[{\cite[Theorem 88]{hw08}}] \label{hw-lemma} 
Suppose that $a, m \in \mathbb N$, with gcd$(a, m) = 1$. Then for any $n \in \mathbb N$, 
\begin{equation*} 
a^n \equiv 1 (\text{mod } m) \quad \text{ if and only if } \quad \text{ord}_m(a) \text{ divides } n.  
\end{equation*} 
\end{lemma} 
\begin{lemma}[{\cite[Lemma 4]{s60}}] \label{Schmidt-lemma}
For every odd integer $r$, there exists $c_0(r) > 0$ such that 
\begin{equation} 
\text{ord}_{2^k}(r) \geq c_0(r) 2^k \quad \text{ for all } k \in \mathbb N. 
\end{equation} 
\end{lemma} 
\noindent We will use these two facts to control the number of summands in $\mathtt I_1$. 
\subsection{Cardinality of $\mathbb V_1(r, M)$}
\begin{lemma} \label{I1-lemma}
For every odd integer $r \geq 3$, there exists a constant $C_0(r) > 0$ with the following property. For any $N \in \mathbb N$ in the range \eqref{N-and-R}, let $\mathbb V_1(r, N)$ be as in \eqref{def-W1}. Then 
\begin{equation} \label{W1-cardinality}
\#(\mathbb V_1(r, N)) \leq C_0(r) N 2^{-\sqrt{R}}. 
\end{equation} 
In particular, Proposition \ref{Ii-lemma-parta} holds. 
\end{lemma}  
\begin{proof} 
Let us observe, from the definition \eqref{def-W1} of $\mathbb V_1$ and Lemma \ref{hw-lemma}, that 
\begin{align*} 
\mathbb V_1(r, N) &= \bigl\{v \in \{1, 2, \ldots, N\} : r^v = 1 \; \text{mod } 2^{R_0} \bigr\} \\
&= \bigl\{v \in \{1, 2, \ldots, N\} : v \text{ is a multiple of } \text{ord}_{2^{R_0}}(r) \bigr\}.
\end{align*}
This means that 
\[ \#\bigl(\mathbb V_1(r, N)\bigr) \leq N\bigl[\text{ord}_{2^{R_0}}(r) \bigr]^{-1} \leq [c_0(r)]^{-1} N2^{-R_0},  \]    
where the last step follows from the lower bound on the order of $r$ given by Lemma \ref{Schmidt-lemma}. Since $R_0 = \lfloor \sqrt{R} \rfloor \geq \sqrt{R} - 1$,
we arrive at the desired inequality \eqref{W1-cardinality} with $C_0(r) = 2/c_0(r)$. 
\vskip0.1in 
\noindent Combining \eqref{W1-cardinality} with the trivial bound $|\widehat{\mu}| \leq 1$ and the range \eqref{N-and-R} of $N$ leads to
\[ \mathtt I_1(h;r, N) \leq  \sum_{u=1}^N \#(\mathbb V_1(r, N)) \leq N \#(\mathbb V_1(r, N)) \leq C_0 2^{2R-\sqrt{R}}.  \]
This establishes \eqref{I1-est}, completing the proof of Proposition \ref{Ii-lemma-parta} and hence of Lemma \ref{I1-lemma}.  
\end{proof}

\section{Estimation of $\mathtt I_2$} \label{I2-lemma-proof-section} 
As mentioned in Section \ref{I-est-section}, the estimation of $\mathtt I_2$ will require a finer analysis of the Fourier coefficients $\widehat{\mu}$. 
Before heading out to the actual estimates, let us first recall the formula for $\widehat{\mu}$ from \eqref{mu-hat-infinite-product} and \eqref{def-El}, and pause for a moment to study the nature of the product $\mathfrak E_{\ell}$ in \eqref{def-El}. It is instructive to consider two extreme cases, which motivate the estimation procedure. 
\begin{itemize} 
\item Suppose that the dyadic digits $\pmb{\eta}_{\ell} := \{\mathtt d_j(\eta) :  j \in \mathcal B_{\ell}\}$ of $\eta$ are all zero, where $\mathcal B_{\ell}$ is the $\ell^{\text{th}}$ block of indices given by \eqref{block-def}. Thus for $j \in \mathcal B_{\ell}+1$, we see that 
\begin{align*}
&\bigl\{2^{-j} \eta \bigr\} = 2^{-j} \sum_{i = 0}^{j-1} \mathtt d_i(\eta) 2^i = \sum_{i = 0}^{K_{\ell-1}} \mathtt d_i(\eta) 2^{i-j} < 2^{K_{\ell-1} + 1 - j},  \text{ which implies } \\  
&\sum_{j \in \mathcal B_{\ell}+1} \bigl\{ 2^{-j} \eta \bigr\}^2 \leq \frac{1}{4}, \; \text{ and therefore } \; \sum_{k \in \mathcal B_{\ell}} \sin^{2} (\pi \eta 2^{-k-1}) \leq C_0, 
\end{align*} 
for some absolute constant $C_0 > 0$. This in turn means that each factor in the product
\begin{equation} \label{cos-product} \bigl| \mathfrak E_{\ell}(\eta) \bigr| = \prod_{k \in \mathcal B_{\ell}} \bigl| \cos (\pi \eta 2^{-k} )\bigr| = \prod_{k \in \mathcal B_{\ell}} \left|1 - 2 \sin^{2}(\pi \eta 2^{-k-1})\right|  \end{equation} 
is close enough to 1 to ensure that $\mathfrak E_{\ell}$ is bounded above and below by a constant independent of $\ell$. The same feature holds if the digits in $\pmb{\eta}_{\ell}$ are all 1. In other words, long unbroken runs of single digits in $\pmb{\eta}_{\ell}$ lead to relatively large values of $\mathfrak E_{\ell}(\eta)$. 
\item In contrast, let us now assume that the digits in $\pmb{\eta}_{\ell}$ alternate, with the even entries being 1.  Then for $k \in \mathcal B_{\ell}$, 
\begin{equation}  
\bigl\{\eta 2^{-k} \bigr\} = \Bigl\{ 2^{-k} \sum_{j=0}^{k-1} \mathtt d_j(\eta) 2^j \Bigr\} = \Bigl\{ 2^{-k} \Bigl[ \sum_{j=0}^{K_{\ell-1}} \mathtt d_j(\eta) 2^j + \sum^{\dagger} 2^{j}\Bigr] \Bigr\}. \label{frac-part}   
\end{equation} 
The sum ${\overset{\dagger}{\sum}} 2^j$ ranges over all  even integers $j$ in $\mathcal B_{\ell}$ with $j < k$. If $j_0$ and $j_1$ are, respectively, the smallest and largest indices such that $2j_0 \geq K_{\ell-1}+1$ and $2(j_1-1) \leq k$, then 
\begin{equation} 2^{-k} {\overset{\dagger}{\sum}} 2^j = 2^{-k} \sum_{j=j_0}^{j_1-1} 2^{2j} = 2^{-k} \frac{4^{j_1} - 4^{j_0}}{3} = \frac{2^{2j_1-k}}{3} - \frac{2^{2j_0-k}}{3}. \end{equation}  
As a result, the sum in \eqref{frac-part} reduces to 
\[ 2^{-k} \Bigl[ \sum_{j=0}^{K_{\ell-1}} \mathtt d_j(\eta) 2^j + \sum^{\dagger} 2^{j} \Bigr] = \frac{2^{2j_1-k}}{3} + \mathfrak e_k, \text{ with } \mathfrak e_k :=  2^{-k} \sum_{j=0}^{K_{\ell-1}} \mathtt d_j(\eta) 2^j  - \frac{2^{2j_0-k}}{3}. \]
The sum in \eqref{frac-part} is therefore  of the form $\frac{2^M}{3} + \mathfrak e_k$ for some $M \in \mathbb N$. The error term $\mathfrak e_k$ obeys 
\[ \bigl| \mathfrak e_k \bigr| \leq 2^{-k} \sum_{j=0}^{K_{\ell-1}} 2^j + \frac{2^{2j_0-k}}{3} \leq 2^{K_{\ell-1}+1-k} + \frac{2^{2j_0-k}}{3} \leq    2^{-4} + \frac{1}{4 \cdot 3} < \frac{1}{6}, \] 
provided $k > K_{\ell-1} +4$. For every such $k \in \mathcal B_{\ell}$, we observe that $\{\eta 2^{-k}\} = \{2^{M}/3 + \mathfrak e_k \}$ is either of the form $\frac{1}{3} + \mathfrak e_k$ or $\frac{2}{3} + \mathfrak e_k$. Thus  
\[ |\cos(\pi \eta 2^{-k})| \leq |\cos({\pi}/{6})| = \frac{\sqrt{3}}{2} \text{ for all indices $k \in \mathcal B_{\ell}$ except possibly the first four.} \]  Inserting this into \eqref{cos-product} gives $|\mathfrak E_l(\eta)| \leq (\frac{\sqrt{3}}{2})^{K_{\ell} - K_{\ell-1}-4}$, a far smaller upper bound compared with the one obtained in the previous case.  
\end{itemize} 
The above analysis suggests that for $\mathfrak E_{\ell}$ to attain values that are significantly smaller than its maximum, consecutive digits in $\pmb{\eta}_{\ell}$ must change frequently.  The estimation procedure of $\mathtt I_2$ employed in this section will make this idea precise. 

 \subsection{The parameters $\mathtt t$ and $\mathtt T$}
Given $R \in \mathbb N$, let us first define two $R$-dependent indices $\mathtt t, \mathtt T \in \mathbb N$, which will play an important role in the estimation of $\mathtt I_2$: 
\begin{equation} \label{t&T}
\sqrt{R} \in (K_{\mathtt t-1}, K_{\mathtt t} ], \qquad R \in (K_{\mathtt T-1}, K_{\mathtt T}], 
\end{equation} 
where $\mathcal K = \{K_{\ell} : \ell \geq 1\}$ is the sequence of positive integers used to define $\mu$; see \eqref{block-def}, \eqref{epsilon-assumption} and \eqref{def-mu} in Section \ref{mu-construction-section}. Clearly $\mathtt t \leq \mathtt T$. In what follows, we will assume that $\mathcal K$ satisfies the growth conditions 
\begin{equation} \label{choice-of-K}
K_{\ell} \geq 10 K_{\ell-1} \quad \text{ and } \quad \mathtt t < \mathtt T-2 \text{ for all large $R \in \mathbb N$},  
\end{equation} 
In particular, these conditions hold for the choice of $\mathcal K$ given in \eqref{choice-of-K-epsilon}, as will be verified in Section \ref{special-choices-section}. The condition \eqref{choice-of-K} implies $K_{\ell} \geq 10^{\ell}$, which places a bound on the growth rate of $\mathtt T$: 
\begin{equation} \label{T-growth} 
10^{\mathtt T-1} \leq K_{\mathtt T-1} \leq R, \quad \text{ or } \quad \mathtt T \leq C_0 \log R.    
\end{equation}    
\vskip0.1in 
\noindent Since each factor in the infinite product \eqref{mu-hat-infinite-product} representing $\mu$ is at most one in absolute value,  we arrive at the following estimate: for $h \in \mathbb Z \setminus \{0\}$,  
\begin{align} \label{mu-hat-infinite-product-2} 
\bigl| \widehat{\mu}\bigl(\eta \bigr)\bigr| &\leq \prod_{\ell=\mathtt t+1}^{\mathtt T-1}  \Biggl| \Bigl[ \varepsilon_{\ell} + (1 - \varepsilon_{\ell}) \prod_{k = K_{\ell-1}+1}^{K_{\ell}} \mathfrak E_{\ell}(\xi) \Bigr] \Biggr|, \text{ where }  \\ 
\eta = \eta(u, v; h, r)  &:= hr^u(r^v-1) \in \mathbb Z \setminus \{0\}, \quad \xi = \xi(u, v;h, r) := \eta \text{ mod } 2^{R}. \label{def-xi}
\end{align} 
Here we have used the fact that the function $\mathfrak E_{\ell}$ is periodic with period $2^{K_{\ell}}$, and therefore also periodic with period $2^{R}$ for $R \geq K_{\ell}$. Thus we have the relation $\mathfrak E_{\ell}(\eta) = \mathfrak E_{\ell}(\xi)$ for every $\ell \in \{\mathtt t+1, \ldots, \mathtt T-1 \}$. 
 It follows from these definitions that the binary digit sequence representing $\xi$ is a truncation of the corresponding digit sequence for $\eta$. Specifically, 
\begin{align} 
&\text{if } \eta = \sum_{k \geq 0} \mathtt d_k(\eta) 2^{k} \text{ with } \mathtt d_k(\cdot) \in \{0, 1\}, \quad \text{ then } \quad  \xi = \sum_{k=0}^{2^{R} - 1} \mathtt d_k(\eta) 2^{k}; \label{eta-and-xi}  \\  
&\text{in other words, } \mathtt d_k(\xi) = \begin{cases} \mathtt d_k(\eta) &\text{ if } 0 \leq k < 2^{R}, \\ 0 &\text{ otherwise. }  \end{cases} \nonumber  
\end{align}  

\subsection{A digit-based decomposition of $\mathtt I_2$}
Let $N$ be any integer in the range \eqref{N-and-R}. The choice of $R$ generates the indices $\mathtt t$ and $\mathtt T$ as in \eqref{t&T}. Pick any index $\ell \in \{\mathtt t+1, \ldots, \mathtt T-1 \}$. For $\xi \in \{0, 1, \ldots, 2^R-1\}$ as in \eqref{def-xi} and \eqref{eta-and-xi}, let us define 
\begin{equation}  \label{def-Nl}
\begin{aligned} 
\mathfrak N_{\ell}(\xi) &:= \# \Bigl\{k \in \mathcal B_{\ell} : \bigl(\mathtt d_{k-2}(\xi), \mathtt d_{k-1}(\xi) \bigr) = (1, 0) \text{ or } (0, 1) \Bigr\} \\ 
&:= \# \Bigl\{k \in \{K_{\ell-1}, \ldots, K_{\ell}-1\} : 1 \leq \mathtt d_{k-1} + 2 \mathtt d_{k} \leq 2 \Bigr\} 
\end{aligned} 
\end{equation} 
In other words, $\mathfrak N_{\ell}(\xi)$ is the number of consecutive digit changes that occur in the $\ell^{\text{th}}$ block 
\begin{equation} \label{bar-block}
\overline{\mathcal B}_{\ell} := \{K_{\ell-1} -1 \leq k \leq K_{\ell}-1 \}
\end{equation} 
of the binary string $\{\mathtt d_k(\xi) : k \geq 0\}$ consisting of the dyadic digits of $\xi$. The size of $\mathfrak N_{\ell}(\xi)$  motivates the following decomposition of $\{1, 2, \ldots, N\}$ depending on $v$:  
\begin{align} 
\mathbb U_1 = \mathbb U_1(v) &:= \left\{u \in \{1, 2, \ldots, N\}\; \Biggl| \; \begin{aligned} &\exists \text{ some } \ell \in \{\mathtt t+1, \ldots, \mathtt T-1 \} \label{def-U1} \\  &\text{such that }  \# \bigl[ \mathfrak N_{\ell}(\xi) \bigr] < \alpha \#(\overline{\mathcal B}_{\ell}) \end{aligned} \right\}, \\
\mathbb U_2 = \mathbb U_2(v) &:= \{1, 2, \ldots, N\} \setminus \mathbb U_1(v). \label{def-U2}
\end{align} 
Here $\alpha \in (0, \frac{1}{4})$ is an absolute constant whose value will be specified shortly, in Lemma \ref{digit-lemma-1} below. 
This leads to a corresponding decomposition of the sum $\mathtt I_2$ defined in \eqref{def-Ii}: 
\begin{equation} 
\mathtt I_2 = \mathtt I_{21} + \mathtt I_{22}, \qquad \mathtt I_{2j} = \mathtt I_{2j}(h;r,N) := \sum_{v \in \mathbb V_2} \sum_{u \in \mathbb U_j(v)} \bigl| \widehat{\mu}\bigl(hr^u(r^v-1) \bigr)\bigr|. \label{def-I2j}
\end{equation} 
The main estimates for $\mathtt I_{21}$ and $\mathtt I_{22}$ are listed here.  
\begin{lemma} \label{I21-lemma}
For any odd integer $r \geq 3$ and $h \in \mathbb Z \setminus \{0\}$, there exist constants $c_0, C_0 > 0$ depending only on $r$ and $h$ with the following property. For any $N$ in the range \eqref{N-and-R} and $v \in \mathbb V_2$ as in \eqref{def-W2},  
\begin{equation} \label{U1-cardinality} 
\# \bigl[ \mathbb U_1(v) \bigr] \leq  C_0 N2^{-c_0 \sqrt{R}}. 
\end{equation}
For $\mathtt I_{21}$ as in \eqref{def-I2j}, this implies
\begin{equation} \label{I21-bound} 
\mathtt I_{21} \leq C_0 N^2 2^{-c_0 \sqrt{R}}. 
\end{equation} 
\end{lemma} 
\noindent In order to estimate $\mathtt I_{22}$, we will need to make some assumptions on the parameters $\mathcal K$ and $\pmb{\varepsilon}$. Let us call $\mathcal K, \pmb{\varepsilon}$ an {\em{admissible choice of parameters with exponent $\gamma$}} if 
\begin{itemize} 
\item $K_{\ell} \nearrow \infty$ with the additional criteria \eqref{choice-of-K}, 
\item $\varepsilon_{\ell} \rightarrow 0$ subject to \eqref{epsilon-assumption}, 
\item $\mathcal K$ and $\pmb{\varepsilon}$ are connected by the following relation: 
\begin{equation} 
\prod_{\ell=\mathtt t+1}^{\mathtt T-1} \varepsilon_{\ell} < K_{\mathtt T}^{-\gamma} \quad \text{ for all sufficiently large } R \in \mathbb N. \label{tT-condition} 
\end{equation} 
\end{itemize} 
Admissibility is a relatively easy requirement to meet. We will see in Section \ref{special-choices-section} that the choice of parameters given in \eqref{choice-of-K-epsilon} is admissible for all $\gamma > 0$. In fact, so is $K_{\ell} = K^{\ell}, \varepsilon_{\ell} = 1/\ell$; the verification is left to the reader. 
\begin{lemma} \label{I22-lemma}
Let $\mathcal K, \pmb{\varepsilon}$ be an admissible choice of parameters with exponent $\gamma > 1$. Then for any odd integer $r \geq 3$ and $h \in \mathbb Z \setminus \{0\}$, there exist constants $c_0, C_0 > 0$ depending only on the aforementioned quanitities  with the following property: for $N$ in the range \eqref{N-and-R}, the sum $\mathtt I_{22}$ given by \eqref{def-I2j} obeys the bound 
\begin{equation} \label{I22-bound}
\mathtt I_{22}(h; r, N) \leq C_0 N^2 \bigl[ R^{-\gamma} + 2^{-c_0 \sqrt{R}} \bigr]. 
\end{equation} 
\end{lemma} 
\noindent Lemmas \ref{I21-lemma} and \ref{I22-lemma} will be proved in Sections \ref{I21-estimation-section} and \ref{I22-estimation-section} respectively. Assuming these, the proof of Proposition \ref{Ii-lemma-partb} is completed as follows. 

\subsection{Proof of Proposition \ref{Ii-lemma-partb}}
Let $\mathcal K, \pmb{\varepsilon}$ be an admissible choice of parameters with exponent $\gamma > 1$. Let $\mu = \mu[\mathcal K, \pmb{\varepsilon}]$ be the corresponding measure constructed in Section \ref{mu-construction-section}. Proposition \ref{Ii-lemma-partb} is a direct consequence of the decomposition \eqref{def-I2j}, followed by the estimates \eqref{I21-bound} and \eqref{I22-bound} derived in Lemmas \ref{I21-lemma} and \ref{I22-lemma} respectively. Combining all these pieces yields
\[ \mathtt I_2 = \mathtt I_{21} + \mathtt I_{22} \leq C_0 N^2 \bigl[2^{-c_0 \sqrt{R}} + R^{-\gamma} \bigr], \]
which is the desired conclusion.  
\qed 
\section{Estimation of $\mathtt I_{21}$} \label{I21-estimation-section}
The goal of this section is to prove Lemma \ref{I21-lemma}, in particular the bound \eqref{U1-cardinality} on $\#[\mathbb U_1(v)]$. The following subsection recalls two number-theoretic lemmas due to Schmidt \cite{s60} which play an important role in this estimate.  
\subsection{Schmidt's lemmas on digits}
In the digit expansion of most numbers, long blocks comprising a single digit are possible but infrequent; most numbers switch digits often. The first lemma makes this precise. 
\begin{lemma}[{\cite[Lemma 3]{s60}}] \label{digit-lemma-1}
There exists a constant $\alpha \in (0, \frac{1}{4})$ with the following property. For all sufficiently large $k$, the number of binary strings $(\mathtt d_0, \mathtt d_1, \ldots, \mathtt d_{k-1}) \in \{0, 1\}^k$ obeying 
\begin{align} \label{changing-digits} 
\# \Bigl\{1 \leq \ell \leq k-1 &: 1 \leq \mathtt d_{\ell-1} + 2 \mathtt d_{\ell} \leq 2 \Bigr\} \nonumber \\ = &\# \Bigl\{1 \leq \ell \leq k-1 : (\mathtt d_{\ell-1}, \mathtt d_{\ell}) = (1, 0) \text{ or } (0, 1) \Bigr\}  \leq \alpha k 
\end{align} 
is at most $2^{\frac{3k}{4}}$. In other words, there exist at most $2^{\frac{3k}{4}}$ integers 
\[ \mathtt D = \sum_{\ell=0}^{k-1} \mathtt d_{\ell} 2^{\ell} \in \{0, 1, \ldots, 2^{k}-1\} \] whose digits $(\mathtt d_0, \ldots, \mathtt d_{k-1})$ in base 2 obey \eqref{changing-digits}. 
\end{lemma}  
\noindent In fact, an explicit value of $\alpha$ obeying Lemma \ref{digit-lemma-1} is provided in \cite{s60}. Specifically, any $\alpha$ with $2^{\frac{1}{8}} > (2 \alpha)^{\alpha} (1 - 2 \alpha)^{\frac{1}{2} - \alpha}$ will do, although this will not be needed in the sequel.  
\vskip0.1in 
\noindent A number of important observations in \cite{s60} quantify a common theme: numbers that are extremely ``non-random'' in a certain base acquire ``random-like'' qualities when viewed in a multiplicatively independent base. Suppose first that $r \sim s$; then it is easy to see that given any integer $k \geq 1$, there is a constant $C_k \geq 1$ such that all except possibly the first $C_k$ of the numbers $\{\varrho r^m : m \geq 0 \}$ are 0 mod $s^k$, for any $\varrho \in \mathbb Z \setminus \{0\}$. A remarkable conclusion of \cite{s60} is that this cannot happen if $r \not\sim s$. For instance, assume that $r$ and $s$ are relatively co-prime, as is the case for us, since $s = 2$ and $r$ is odd. In this case, the numbers $\{\varrho r^m : 0 \leq m < s^k \}$ are roughly equally distributed among the residue classes of $s^k$. In other words, not too many of these numbers can fall in the same residue class mod $s^k$. The following precise version of this statement for $s = 2$ is relevant for us.  
\begin{lemma}[{\cite[Lemma 5A, $s=2$]{s60}}]  \label{digit-lemma-2} 
Given any odd integer $r \in \mathbb N \setminus \{1\}$, there exists a constant $c(r) > 0$ as follows. 
\vskip0.1in
\noindent For any integer $\varrho \in \{\pm 2, \pm 3, \ldots \}$, let us consider the factorization 
\begin{equation} 
\varrho = \varrho_2 \varrho_2' \quad \text{ where } \varrho_2, \varrho_2' \in \mathbb N, \; \varrho_2 \in \bigl\{2^n: n = 0, 1, 2, \ldots \bigr\}, \; 2 \nmid \varrho_2'. 
\end{equation} 
Then for every $k \geq 1$ and for every $\sigma \in \{0, 1, \ldots, 2^k-1 \}$, 
\begin{equation}  \label{digit-lemma-2-ineq} 
\# \bigl\{m \in \{0, 1, \ldots, 2^k-1\} : \varrho r^m = \sigma \text{ mod } 2^k \bigr\} \leq c(r) \varrho_2.  
\end{equation}  
\end{lemma} 
\subsection{Cardinality of digit sequences with small variations}
The main ingredient in the proof of Lemma \ref{I21-lemma} is an estimate on the number of binary digit sequences without too many changes in consecutive digits. The goal of this subsection is to prove this estimate, in Lemma \ref{U1tau-cardinality} below. To set up the necessary tools, let us recall the definition of $\mathbb V_2$ from \eqref{def-W2}, the parameters $\mathtt t$ and $\mathtt T$ from \eqref{t&T}, the integer $\xi$ from \eqref{def-xi} and \eqref{eta-and-xi}, and the collection $\mathfrak N_{\ell}(\xi)$ from \eqref{def-Nl}. Let us fix indices 
\begin{align} \label{choice-of-tau}  
&v \in \mathbb V_2 \text{ and } \ell \in \{ \mathtt t + 1, \ldots, \mathtt T-1 \}, \text{ and define } \\  
\label{def-U1vtau}
&\mathbb U_1(v, \ell) := \bigl\{u \in \{1, \ldots, N\}: \#\bigl[ \mathfrak N_{\ell}(\xi) \bigr] < \alpha \#(\overline{\mathcal B}_{\ell}) \bigr\}, 
\end{align} 
with $\overline{\mathcal B}_{\ell}$ as in \eqref{bar-block}. Here and henceforth, $\alpha \in (0, \frac{1}{4})$ denotes the fixed constant whose existence has been established in Lemma \ref{digit-lemma-1}. While $\mathbb U_1(v, \ell)$ depends implicitly on $h$ and $r$ through $\xi$, these two quantities remain fixed throughout, and we omit them from the notation. 
\begin{lemma} \label{U1vtau-cardinality-lemma}
There exist constants $c, C_0 > 0$ such that for $N$ in the range \eqref{N-and-R} and all $v, \ell$ as above,  
\begin{equation} \label{U1tau-cardinality}
\# \bigl[ \mathbb U_1(v, \ell) \bigr] \leq  C_0 N 2^{-c \sqrt{R}}. 
\end{equation} 
\end{lemma} 
\begin{proof} 
We will estimate the cardinality of $\mathbb U_1(v, \ell)$ in two steps. The first step is to identify how many distinct integers $\xi \in \{0, 1, \ldots, 2^R - 1\}$ obey the defining condition of $\mathbb U_1(v; \ell)$. The next step is to estimate, for a given $v$ and $\ell$, the number of indices $u$ that generate a single element $\xi$ with this property via the formulae \eqref{def-xi} and \eqref{eta-and-xi}. With this in mind, let us define 
\[ \mathbb W^{\ast}(\ell) := \bigl\{ \xi \in \{0, 1, \ldots, 2^{R}-1\} : \# \bigl[ \mathfrak N_{\ell}(\xi) \bigr] < \alpha \#(\overline{\mathcal B}_{\ell}) \bigr\}. \] 
Recalling from \eqref{bar-block} the role of the block $\overline{\mathcal B}_{\ell} := \{K_{\ell-1}-1, K_{\ell-1}, \ldots, K_{\ell}-1 \}$ in the definition \eqref{def-Nl} of $\mathfrak N_{\ell}(\xi)$, let us write an element $\xi \in \mathbb W^{\ast}(\ell)$ in the form
\begin{equation} \label{def-xil'} \xi = \sum_{k=0}^{R-1} \mathtt d_k(\xi) 2^k = \xi_{\ell} + \xi_{\ell}' \quad \text{ where } \quad \xi_{\ell}  = \sum_{k \in \overline{\mathcal B}_{\ell}} \mathtt d_k(\xi) 2^k \quad \text{ and } \quad \xi_{\ell}' = \xi - \xi_{\ell}. 
\end{equation} 
Then Lemma \ref{digit-lemma-1} dictates the maximum number of distinct choices of ordered strings $(\mathtt d_k(\xi) : k \in \overline{\mathcal B}_{\ell} )$, which is the same as the maximum number of distinct choices of $\xi_{\ell}$. Specifically, 
\begin{align*} 
\#\{\xi_{\ell} : \xi \in \mathbb W^{\ast}(\ell) \} &= \#\Bigl\{(\mathtt d_k(\xi): k \in \overline{\mathcal B}_{\ell}) :  \# \bigl[ \mathfrak N_{\ell}(\xi) \bigr] < \alpha \#(\overline{\mathcal B}_{\ell})   \Bigr\} \\ 
&\leq 2^{\frac{3}{4} \#(\overline{\mathcal B}_{\ell})} = 2^{\frac{3}{4}(K_{\ell} - K_{\ell-1}+1)}. \end{align*}  
On the other hand, it follows from the definition \eqref{def-xil'} that the digits of $\xi_{\ell}'$ are all zero on the block $\overline{\mathcal{B}}_{\ell}$, whereas each of its other $R - \#(\overline{\mathcal B}_{\ell})$ digits trivially admits at most two choices 0 and 1. The total number of possible choices of $\xi_{\ell}'$ is therefore $2^{R - \#(\overline{\mathcal B}_{\ell})}$. These observations provide the following bound on the cardinality of $\mathbb W^{\ast}$:  
\begin{equation}  
\# \bigl[\mathbb W^{\ast}(\ell) \bigr] \leq 2^{\frac{3}{4} (K_{\ell} - K_{\ell-1}+1)} 2^{R - (K_\ell - K_{\ell-1} + 1)} < 2^{R - \frac{1}{4}(K_{\ell} - K_{\ell-1})}.   \label{Wstar-cardinality} 
\end{equation}   
The set $\mathbb W^{\ast}(\ell)$ leads to a natural decomposition of $\mathbb U_1$:  
\begin{align} 
&\mathbb U_1(v, \ell) = \bigsqcup \bigl\{ \mathbb U_{1 \xi}(v; \ell) : \xi \in \mathbb W^{\ast}(\ell) \bigr\}, \quad \text{ where } \label{U1-decomp} \\ &\mathbb U_{1 \xi}(v, \ell) := \bigl\{u \in \mathbb U_1(v;\ell): \eta = hr^{u}(r^v-1) \equiv \xi \text{ mod } 2^R \bigr\}. \nonumber
\end{align} 
We will use Lemma \ref{digit-lemma-2} to bound the cardinality of $\mathbb U_{1\xi}$. In the notation of Lemma \ref{digit-lemma-2}, set
\[ \varrho = h(r^v-1), \quad \text{ so that } \quad \varrho_2 = h_2 (r^v-1)_2.  \] 
The defining condition \eqref{def-W2} of $\mathbb V_2$ ensures that $2^{R_0}$ does not divide $(r^v-1)$, i.e., $(r^v-1)_2 < 2^{R_0}$.  Substituting this into \eqref{digit-lemma-2-ineq} leads to 
\begin{align} 
\# \bigl[ \mathbb U_{1\xi}(v, \ell)\bigr] &\leq \# \bigl\{u \in \{1, 2, \ldots, 2^{R}\}: \varrho r^u = \xi \text{ mod } 2^R \bigr\}  \nonumber \\
&\leq c(r) \varrho_2 = c(r) h_2 (r^v-1)_2 \leq c(r) h 2^{R_0} \leq C_0 2^{\sqrt{R}}. \label{U1xi-cardinality} 
\end{align} 
Combining \eqref{U1-decomp}, \eqref{Wstar-cardinality} and \eqref{U1xi-cardinality}, we arrive at 
\begin{align*} 
\#(\mathbb U_1(v, \ell)) &\leq \#\bigl[\mathbb W^{\ast} (\ell) \bigr] \max_{\xi} \bigl[ \#(\mathbb U_{1\xi}) \bigr] \leq C_0 2^{R + \sqrt{R} - \frac{1}{4}(K_{\ell} - K_{\ell-1})} \\
&\leq C_0 N 2^{- c \sqrt{R}}.   
\end{align*}   
The last step follows from the growth condition \eqref{choice-of-K}, the choice \eqref{choice-of-tau} of $\ell > \mathtt t$ and the definition \eqref{t&T} of $\mathtt t$, via the inequality 
\[ \sqrt{R} - \frac{1}{4}(K_{\ell} - K_{\ell-1}) \leq  \sqrt{R} - \frac{9}{40}K_{\ell} \leq - c \sqrt{R}. \]  
Here $c = 5/4$. This establishes \eqref{U1tau-cardinality}, 
\end{proof} 
\subsection{Proof of Lemma \ref{I21-lemma}} 
\begin{proof} 
The set of integers $\mathbb U_1(v)$ defined in \eqref{def-U1} admits a natural decomposition: 
\[ \mathbb U_1(v) = \bigcup_{\ell = \mathtt t+1}^{\mathtt T-1} \mathbb U_1(v, \ell),  \]
with $\mathbb U_1(v, \ell)$ as in \eqref{def-U1vtau}.  Lemma \ref{U1vtau-cardinality-lemma} yields constants $c, C_0 > 0$ such that  
\begin{align*}  \# \bigl[ \mathbb U_1(v) \bigr] &\leq \sum_{\ell=\mathtt t+1}^{\mathtt T-1}  \# \bigl[ \mathbb U_1(v, \ell) \bigr] \leq  \sum_{\ell=\mathtt t+1}^{\mathtt T-1} C_0 N 2^{-c\sqrt{R}} \\
&\leq C_0 \mathtt T N 2^{-c \sqrt{R}} \leq C_0 N (\log R) \, 2^{-c \sqrt{R}} \leq C_0 N 2^{-c_0\sqrt{R}}, 
\end{align*}  
with $c_0 = c/2$. The bound on $\mathtt T$ in the penultimate step above follows from the estimate \eqref{T-growth}. This establishes \eqref{U1-cardinality}.  
 \vskip0.1in
\noindent The inequality \eqref{I21-bound} is an easy consequence of \eqref{U1-cardinality}. Replacing each summand of $\mathtt I_{21}$ in \eqref{def-I2j} by the trivial bound 1, and using \eqref{U1-cardinality}, we obtain 
\[ \mathtt I_{21} \leq \#(\mathbb V_2) \max_{v} \bigl[ \#(\mathbb U_1(v)) \bigr] \leq C_0 N^2 2^{-c_0 \sqrt{R}}.\]
This completes the proof of Lemma \ref{I21-lemma}.  
\end{proof}

\section{Estimation of $\mathtt I_{22}$} \label{I22-estimation-section} 
We prove Lemma \ref{I22-lemma} in this section. The quantity to be analyzed is $\mathtt I_{22}$ given by \eqref{def-I2j}, with $\mathbb V_2$ and $\mathbb U_2(v)$ as in \eqref{def-W2} and \eqref{def-U2} respectively, and $N$ in the range \eqref{N-and-R}.   
\subsection{The Fourier coefficient of $\mu$: from products to sums} 
The size of $\widehat{\mu}(\xi)$ when $v \in \mathbb V_2, u\in \mathbb U_2(v)$is critical for the estimation of $\mathtt I_{22}$. In order to make this precise, let us recall the pointwise bound on $\widehat{\mu}$ given by \eqref{mu-hat-infinite-product-2}, with $\mathtt t$ and $\mathtt T$ being the two $R$-dependent indices defined in \eqref{t&T}. Expanding the product leads to the estimate:
\begin{align} 
\bigl| \widehat{\mu}\bigl(hr^u(r^v-1) \bigr)\bigr| &\leq \prod_{\ell=\mathtt t+1}^{\mathtt T-1} \bigl[ \varepsilon_{\ell} + (1 - \varepsilon_{\ell}) |\mathfrak E_{\ell}(\xi)| \bigr] \nonumber \\
&\leq  \sum_{\mathbf a} \prod_{\ell=\mathtt t+1}^{\mathtt T-1} \varepsilon_{\ell}^{1 - \mathtt a_{\ell}} \bigl[(1 - \varepsilon_{\ell}) |\mathfrak E_{\ell}(\xi)|\bigr]^{\mathtt a_{\ell}}. \label{sum} 
\end{align}  
The summation $\sum_{\mathbf a}$ in \eqref{sum} ranges over all multi-indices $\mathbf a = (\mathtt a_{\mathtt t+1}, \ldots, \mathtt a_{\mathtt T-1}) \in \{0,1\}^{\mathtt T-\mathtt t-1}$.  Inserting the expression in \eqref{sum} into \eqref{def-I2j} results in a further decomposition of $\mathtt I_{22}$: 
\begin{align}
&\mathtt I_{22} (h, r; N) \leq \sum_{\mathbf a }\mathtt I^{\ast}_{\mathbf a} \leq \mathtt J_0 + \mathtt J_1, \text{ where } \label{I22-decomp} \\ 
&\mathtt I^{\ast}_{\mathbf a} = \mathtt I^{\ast}_{\mathbf a}(h, r; N) := \sum_{u,v}^{\ast}
\prod_{\ell=\mathtt t+1}^{\mathtt T-1} \varepsilon_{\ell}^{1 - \mathtt a_{\ell}} \bigl[(1 - \varepsilon_{\ell}) |\mathfrak E_{\ell}(\xi)|\bigr]^{\mathtt a_{\ell}}, \nonumber \\
&\mathtt J_0 = \mathtt J_0(h, r; N) := \mathtt I_{\mathbf 0}^{\ast} = 
\sum_{u,v}^{\ast} \prod_{\ell=\mathtt t+1}^{\mathtt T-1} \varepsilon_{\ell}, \label{def-J0} \\ &\mathtt J_1 = \mathtt J_1(h, r;N) := 
\sum_{u,v}^{\ast}
\sum_{\mathbf a \ne 0} \prod_{\ell=\mathtt t+1}^{\mathtt T-1} |\mathfrak E_{\ell}(\xi)|^{\mathtt a_{\ell}}. \label{def-J1} 
\end{align} 
The summation ${\overset{\ast}\sum}$ occurring in the definitions of $\mathtt I^{\ast}_{\mathbf a}$, $\mathtt J_0$, $\mathtt J_1$ ranges over the collection of indices 
\begin{equation} \label{def-Gamma}
\Gamma := \{(u,v) : v \in \mathbb V_2, u \in \mathbb U_2(v)\}.
\end{equation}  
with $\mathbb V_2$ and $\mathbb U_2(v)$ as in \eqref{def-W2} and \eqref{def-U2} respectively.  
\subsection{A bound on $\mathtt J_0$} 
\begin{lemma} \label{J0-lemma}
Suppose that $\mathcal K$ and $\pmb{\varepsilon}$ obey the condition \eqref{tT-condition}. For $\mu = \mu[\mathcal K, \pmb{\varepsilon}]$  defined by \eqref{def-mu}, let $\mathtt J_0$ be the sum given by \eqref{def-J0}. Then, for any $r \in \mathbb N \setminus \{1\}$ and $h \in \mathbb Z \setminus \{0\}$, there exists a constant $C_0 = C_0(h, r) > 0$ such that for 
 $N$ in the range \eqref{N-and-R}, 
\begin{equation} \label{J0-bound}
\mathtt J_0(h, r; N) \leq N^2 R^{-\gamma} \text{ for all } R \in \mathbb N. 
\end{equation} 
\end{lemma}
\noindent {\em{Remark: }} Lemma \ref{J0-lemma} does not require $r$ to be odd; any integer $r \geq 2$ suffices. 
\begin{proof} 
As we see from the definitions \eqref{N-and-R} and \eqref{t&T}, the indices $\mathtt t$ and $\mathtt T$ depend only on $N$ and $\mathcal K$. Since $\{\varepsilon_{\ell} : \mathtt t < \ell \leq \mathtt T-1 \}$ is independent of $u$ and $v$, we observe that $\mathtt J_0$ is a sum of constant terms. The assumption \eqref{tT-condition} then implies that for all sufficiently large $N$, 
\[\mathtt J_0 \leq \sum_{u=1}^{N} \sum_{v \in \mathbb V_2} K_{\mathtt T}^{-\gamma} \leq N \# \bigl(\mathbb V_2 \bigr) R^{-\gamma} \leq N^2 R^{-\gamma}, \] 
where the penultimate step follows from the choice of $\mathtt T$ in \eqref{t&T}. This is the conclusion \eqref{J0-bound}. 
\end{proof}

\subsection{A bound on $\mathtt J_1$} 
Let us observe from \eqref{def-J1} that each summand of $\mathtt J_1$ involves a product of factors $\mathfrak E_{\ell}(\xi)$ defined in \eqref{def-El}. Each factor $\mathfrak E_{\ell}$, in turn, is a product of simple trigonometric functions. In order to estimate $\mathtt J_1$, let us start with an estimate on these fundamental building blocks of  $\mathfrak E_{\ell}$. 
\begin{lemma} \label{cos-lemma}
Let us fix any $\varrho \in \mathbb N$ whose binary digits are $\{\mathtt d_{\ell}(\varrho): \ell \geq 0\}$: 
\begin{equation} \label{rho-expansion}   
\varrho = \sum_{\ell \geq 0} \mathtt d_\ell(\varrho) 2^{\ell}, \quad \mathtt d_{\ell}(\cdot) \in \{0,1\}. 
\end{equation} 
Assume there exists an index $k \in \mathbb N$ such that $ \bigl(\mathtt d_{k-1}(\varrho), \mathtt d_k(\varrho) \bigr) \ne (0,0)$ or $(1,1)$. Then  
\begin{equation} \label{cos-small}
\bigl| \cos \bigl(\pi \varrho 2^{-k-1}\bigr) \bigr| = \Bigl| \frac{1}{2} \Bigl(1 + e(\varrho 2^{-k-1})\Bigr) \Bigr| \leq \cos \Bigl(\frac{\pi}{4}\Bigr) = \frac{\sqrt{2}}{2}<1.   
\end{equation} 
\end{lemma} 
\begin{proof} 
Since $|\cos(\cdot)|$ is $\pi$-periodic, proving \eqref{cos-small} requires a bound on the fractional part of $\varrho 2^{-k-1}$. It follows from the digit expansion \eqref{rho-expansion} of $\varrho$ that  
\begin{align*} 
\left\{ \frac{\varrho}{2^{k+1}} \right\}  &= \sum_{\ell=0}^k \mathtt d_{\ell}(\varrho) 2^{\ell - (k+1)}  \\ 
&= \sum_{\ell=0}^{k-2} \mathtt d_{\ell}(\varrho) 2^{\ell - (k+1)} + \mathtt d_{k-1}(\varrho) 2^{-2} + \mathtt d_k(\varrho) 2^{-1} \\ 
&\leq 2^{-(k+1)} \sum_{\ell=0}^{k-2} 2^{\ell} + \begin{cases} \frac{1}{2} &\text{ if } (\mathtt d_{k-1}, \mathtt d_k) = (0,1), \\ 
\frac{1}{4} &\text{ if }  (\mathtt d_{k-1}, \mathtt d_k) = (1,0),
 \end{cases}  \\
 &\leq \begin{cases}
 \frac{3}{4} - 2^{-k-1} &\text{ if }  (\mathtt d_{k-1}, \mathtt d_k) = (0,1), \\ 
 \frac{1}{2} - 2^{-k-1} &\text{ if }  (\mathtt d_{k-1}, \mathtt d_k) = (1,0).
 \end{cases} 
\end{align*} 
Combining the two estimates above, we find that 
\[ \frac{1}{4} \leq \left\{ \frac{\varrho}{2^{k+1}} \right\} \leq \frac{3}{4}, \quad \text{ which implies } \quad \bigl|\cos(\pi \varrho 2^{-k-1})\bigr| \leq \cos(\pi/4) = \frac{\sqrt{2}}{2} < 1. \] 
This is the bound claimed in \eqref{cos-small}.
\end{proof} 
\noindent With Lemma \ref{cos-lemma} behind us, we are now ready to estimate $\mathtt J_1$. 
\begin{lemma} \label{J1-lemma}
Let $N$, $\Gamma$ and $\mathtt J_1$ be as in \eqref{N-and-R}, \eqref{def-Gamma} and \eqref{def-J1} respectively.  If $(u, v) \in \Gamma$ and $\xi$ is as in \eqref{def-xi}, then 
\begin{equation} \label{El-bound} \bigl| \mathfrak E_{\ell}(\xi) \bigr| \leq \left( \frac{\sqrt 2}{2}\right)^{\alpha(K_{\ell} - K_{\ell-1})} \text{ for all } \ell \in \{\mathtt t+1, \ldots, \mathtt T-1\}, \end{equation} 
where $\mathfrak E_{\ell}(\cdot)$ has been defined in \eqref{def-El}. As a result, 
\begin{equation}  \label{J1-bound}
\mathtt J_1(h, r; N)  \leq C_0 N^2 2^{-c_0 \sqrt{R}}. 
\end{equation}  
 \end{lemma} 
\begin{proof}
Since $(u,v) \in \Gamma$, the defining condition \eqref{def-U2} for $\mathbb U_2(v)$ implies that  
\[ \# \bigl[ \mathfrak N_{\ell}(\xi)\bigr] \geq \alpha \#(\overline{\mathcal B}_{\ell}) \text{ for all } \ell \in \{\mathtt t+1, \ldots, \mathtt T-1\}.  \] 
In other words, the number of indices $k \in \mathcal B_{\ell}$ for which $(\mathtt d_{k-2}(\xi), \mathtt d_{k-1}(\xi)) = (1,0)$ or $(0,1)$ is at least $\alpha(K_{\ell} - K_{\ell-1})$. In view of \eqref{cos-small},
\begin{align} 
&\Bigl| \frac{1}{2} \Bigl(1 + e(h2^{-k})\Bigr) \Bigr|  \leq \frac{\sqrt{2}}{2}  \quad \text{ for all such $k$, and therefore } \nonumber \\   
&\bigl|\mathfrak E_{\ell}(\xi) \bigr| = \prod_{k \in \mathcal B_{\ell}} \Bigl| \frac{1}{2} \bigl(1 + e(2^{-k} \xi )\bigr) \Bigr| \leq \left( \frac{\sqrt{2}}{2}\right)^{\alpha(K_{\ell} - K_{\ell-1})}. \label{El-small}
\end{align} 
This proves \eqref{El-bound}.
 \vskip0.1in
\noindent The growth condition \eqref{choice-of-K} on $K_{\ell}$ places a lower bound on the exponent in \eqref{El-small}: 
\[ \alpha(K_{\ell} - K_{\ell-1}) \geq \frac{9\alpha K_{\ell}}{10} \geq 9\alpha K_{\mathtt t} \text{ for all } \ell \in \{\mathtt t+1, \ldots, \mathtt T-1\},  \] 
from which we obtain the estimate
\begin{equation} \label{El-small-1}
\bigl|\mathfrak E_{\ell}(\xi) \bigr| \leq \kappa, \quad \text{ with } \quad \kappa := \left(\frac{\sqrt{2}}{2} \right)^{9 \alpha K_{\mathtt t}} = 2^{-c K_{\mathtt t}}; \text{ for some }  c > 0.  
\end{equation}
We can choose $c = 9 \alpha/2$. 
\vskip0.1in 
\noindent Using \eqref{El-small-1}, \eqref{t&T} and \eqref{T-growth}, we estimate the following sum: 
\begin{align*}
\sum_{\mathbf a \ne 0} \prod_{\ell = \mathtt t +1}^{\mathtt T-1} \bigl|\mathfrak E_{\ell}(\xi) \bigr|^{a_{\ell}} \leq \sum_{\mathbf a \ne 0} \kappa^{|\mathbf a|} &= \sum_{m=1}^{\mathtt T-\mathtt t-1} {{\mathtt T-\mathtt t-1}\choose{m}} \kappa^m = (1 + \kappa)^{\mathtt T- \mathtt t-1} -1 \\
&\leq C_0 (\mathtt T- \mathtt t-1) \kappa \leq C_0 \log(R) 2^{- c \sqrt{R}} \leq C_0 2^{- c_0 \sqrt{R}}.
\end{align*}
Substituting this into \eqref{def-J1} yields 
\[ \mathtt J_1(h, r; N) \leq \sum^{\ast} \sum_{\mathbf a \ne 0} \prod_{\ell = \mathtt t +1}^{\mathtt T-1} \bigl|\mathfrak E_{\ell}(\xi) \bigr|^{a_{\ell}} \leq C_0 2^{-c_0 \sqrt{R}} \#(\Gamma) \leq C_0 N^2 2^{- c_0 \sqrt{R}},\]
as claimed in \eqref{J1-bound}. This completes the proof of Lemma \ref{J1-lemma}.    
\end{proof} 

\subsection{Proof of Lemma \ref{I22-lemma}} 
The main ingredients of the proof have been collected in the previous subsections. The admissibility assumption on $\mathcal K, \pmb{\varepsilon}$ made in Lemma \ref{I22-lemma} ensures that \eqref{tT-condition} holds for some $\gamma > 1$. This permits the application of Lemma \ref{J0-lemma}.  Combining the decomposition \eqref{I22-decomp} with the estimates \eqref{J0-bound} and \eqref{J1-bound} from Lemmas \ref{J0-lemma} and \ref{J1-lemma} respectively, we arrive at 
\[ \mathtt I_{22} \leq \mathtt J_0 + \mathtt J_1 \leq C_0 2^{2R} \bigl[ R^{-\gamma} + 2^{-c_0 \sqrt{R}}\bigr]. \] 
This is the estimate claimed in \eqref{I22-bound}. 
\qed
\section{A special choice of parameters $\mathcal K$ and $\pmb{\varepsilon}$} \label{special-choices-section} 
It remains to establish Lemma \ref{special-choices-lemma}; namely, if $\mathcal K$ and $\pmb{\varepsilon}$ are chosen to satisfy \eqref{choice-of-K-epsilon}, then the conclusions of Proposition \ref{Ii-lemma-partb} are met. We have seen in Section \ref{I2-lemma-proof-section} that Proposition \ref{Ii-lemma-partb} follows Lemmas \ref{I21-lemma} and \ref{I22-lemma}. Lemma \ref{I21-lemma} does not require any special assumptions on $\mathcal K$ and $\pmb{\varepsilon}$ beyond \eqref{block-def} and \eqref{epsilon-assumption}, whereas Lemma \ref{I22-lemma} requires $\mathcal K$ and $\pmb{\varepsilon}$ to be admissible with some exponent $\gamma > 1$, according to the definition on page \pageref{tT-condition}. In view of this, Lemma \ref{special-choices-lemma} is a consequence of the following result. Proving it is the main objective of this section. 
\begin{lemma} \label{K-epsilon-verification-lemma} 
The choice \eqref{choice-of-K-epsilon} of parameters $\mathcal K$ and $\pmb{\varepsilon}$ is admissible for every $\gamma > 1$. Consequently, Lemma \ref{special-choices-lemma} holds.  
\end{lemma} 
\subsection{A function of slow growth}
\begin{lemma} \label{omega-lemma} 
The non-decreasing function 
\begin{equation}  \omega(x) := \lfloor  \sqrt{\log x} \rfloor, \qquad x > 0 \label{def-omega}  \end{equation}  
obeys the following ``slow growth'' property. Namely, for every $\mathtt M \in (1, \infty)$ and every $\tau > 0$, there exists $\mathtt X = \mathtt X(\mathtt M, \tau)$ such that 
\begin{equation} \label{omega-inequalities} 
\omega(x) \leq \omega(\mathtt M x) \leq (1 + \tau) \omega(x) \quad \text{ for all } x \geq \mathtt X. 
\end{equation}  
\end{lemma} 
\begin{proof}
The first inequality in \eqref{omega-inequalities} is trivial for any $N > 1$, since $\omega$ is a non-decreasing function. For the second inequality, we observe that 
\[ \frac{\omega(\mathtt M x)}{\omega(x)} \leq \frac{\sqrt{\log (\mathtt M x)}}{\sqrt{\log x}-1} \leq \frac{\sqrt{\log x + \log \mathtt M}}{\sqrt{\log x} - 1}  \rightarrow 1 \text{ as } x \rightarrow \infty, \]
which leads to the desired conclusion. 
\end{proof} 
\begin{corollary} \label{Omega-cor} 
Let $\Omega(x) := x \omega(x)$, where $\omega$ is any function on $(0,\infty)$ obeying \eqref{omega-inequalities}. Then for every $N \in (1, \infty)$ and $\kappa > 0$, the inequality
\begin{equation} \label{Omega-inequality}
\Omega \bigl(N(1 - \kappa)x \bigr) < N \Omega(x) \text{ holds for all sufficiently large $x > 0$.}
\end{equation}  
In particular, \eqref{Omega-inequality} holds for $\omega$ as in \eqref{def-omega}.
\end{corollary}
\begin{proof}
The inequality \eqref{Omega-inequality} is the the same as
\[ \Omega \bigl(N(1 - \kappa) x\bigr) = N(1-\kappa)x \omega(N(1 - \kappa)x)  < Nx \omega(x) = N \Omega(x), \]  which simplifies to  
\[ \frac{\omega \bigl(N(1- \kappa)x \bigr)}{\omega(x)} < \frac{1}{1 - \kappa}. \]
Setting $\mathtt M = N(1 - \kappa)$ and $\tau = {\kappa}/{(1 - \kappa)}$ in Lemma \ref{omega-lemma} confirms that the above inequality holds for all sufficiently large $x$. 
\end{proof}

\subsection{Size of $\mathtt t$ and $\mathtt T$ for slow-growing $\omega$} 
\begin{lemma} \label{special-tT-lemma} 
Suppose that $K_{\ell}$ is given by \eqref{choice-of-K-epsilon}, and that for any $R \in \mathbb N$, the indices $\mathtt t$ and $\mathtt T$ are defined by the conditions \eqref{t&T}. Then for every $\kappa > 0$, 
\begin{equation} \label{tT-similar}
2 (\mathtt t - 1)(1 - \kappa) < \mathtt T < 2\mathtt t+1 \quad \text{ for all sufficiently large $R$.}
\end{equation} 
As a result, there exist absolute constants $0< c_0 < 1 < C_0$ such that 
\begin{equation} \label{T-t} 
c_0\mathtt T \leq \mathtt T-\mathtt t \leq C_0 \mathtt T \quad  \text{ for all } R \in \mathbb N. 
\end{equation} 
\end{lemma} 
\begin{proof} 
For any choice of $R \in \mathbb N$, the defining conditions of $\mathtt t$ and $\mathtt T$ given by \eqref{t&T} yield
\[ K_{\mathtt t-1}^2 < R \leq K_{\mathtt t}^2 \quad \text{ and } \quad K_{\mathtt T-1} < R \leq K_{\mathtt T}. \] 
Combining the two inequalities above leads to
\[
K_{\mathtt t-1}^2 = K^{2 \Omega(\mathtt t -1)} < K_{\mathtt T} = K^{\Omega(\mathtt T)} \quad \text{ and } \quad 
K_{\mathtt T-1} = K^{\Omega(\mathtt T-1)} < K_{\mathtt t}^2 = K^{2 \Omega(\mathtt t)}, 
\]
where $\Omega(x) := x \omega(x)$, with $\omega$ as in \eqref{def-omega}. In other words, 
\begin{equation}  \label{2-inequalities}
2 \Omega(\mathtt t -1) < \Omega(\mathtt T) \quad \text{ and } \quad \Omega(\mathtt T-1) < 2 \Omega(\mathtt t). 
\end{equation}  
The first inequality in \eqref{2-inequalities} combined with Corollary \ref{Omega-cor} says that for every $\kappa > 0$, 
\begin{equation}  \label{T-t-left}
\Omega(2(\mathtt t-1)(1 - \kappa)) \leq 2 \Omega(\mathtt t-1) < \Omega(\mathtt T). 
\end{equation} 
The second inequality in  \eqref{2-inequalities} gives, in view of the non-decreasing nature of $\omega$,  
\begin{equation} \label{T-t-right}  
\Omega(\mathtt T -1) < 2 \Omega(\mathtt t) = 2 \mathtt t \omega(\mathtt t)  \leq 2 \mathtt t \omega(2 \mathtt t) = \Omega(2\mathtt t). 
\end{equation}  
Since $\Omega$ is strictly increasing, the inequalities \eqref{T-t-left} and \eqref{T-t-right} lead to the conclusion \eqref{tT-similar}. The deduction of \eqref{T-t} from \eqref{tT-similar} is left to the reader. 
\end{proof} 
\subsection{Proof of Lemma \ref{K-epsilon-verification-lemma}}
Let us assume the parameters $K_{\ell}$ and $\varepsilon_{\ell}$ have been chosen according to \eqref{choice-of-K-epsilon}. In order to verify the admissibility criteria \eqref{tT-condition} on page \pageref{tT-condition}, let us observe that 
\[ \frac{K_{\ell}}{K_{\ell-1}} = K^{(\ell + 1) \omega(\ell+1) - \ell \omega(\ell)} \geq K \geq 10 \] 
if $K$ is chosen sufficiently large. By Lemma \ref{tT-similar}, we also obtain that $\mathtt T - \mathtt t \geq c_0 \mathtt T > 2$ for sufficiently large $R$. Thus both conditions in \eqref{choice-of-K} are verified. Further, the sequence $\varepsilon_{\ell} \rightarrow 0$ and has a divergent sum, confirming \eqref{epsilon-assumption}.  
\vskip0.1in
\noindent It remains to verify \eqref{tT-condition}. For this, we observe that   
\[ \prod_{\ell=\mathtt t+1}^{\mathtt T-1 } \varepsilon_{\ell} = \prod_{\ell=\mathtt t+1}^{\mathtt T-1 } \ell^{-1} < \mathtt t^{-(\mathtt T- \mathtt t - 1 )} \quad \text{ and } \quad K_{\mathtt T} \leq K^{\mathtt T \sqrt{\log \mathtt T}}, \]
the desired condition \eqref{tT-condition} would follow from the inequality 
\begin{equation} 
K^{\gamma \mathtt T \sqrt{\log \mathtt T}} \leq \mathtt t^{\mathtt T - \mathtt t - 1}, \quad \text{ or equivalently } \quad \gamma \mathtt T \sqrt{\log \mathtt T} \log K \leq (\mathtt T - \mathtt t - 1) \log \mathtt t. 
\end{equation}    
In view of the relations \eqref{tT-similar}, \eqref{T-t} between $\mathtt t$ and $\mathtt T$ established in Lemma \ref{special-tT-lemma}, the last estimate is a consequence  of
\[\gamma \mathtt T \sqrt{\log \mathtt T} \log K \leq \frac{c_0}{2}  \mathtt T \log \bigl((\mathtt T-1)/{2} \bigr), \quad \text{ i.e. } \quad \gamma \log K \sqrt{\log \mathtt T} \leq \frac{c_0}{2} \log \bigl((\mathtt T-1)/{2} \bigr). \]
For any $\gamma > 1$, the above inequality holds for all sufficiently large $\mathtt T$, completing the proof.  
\qed

\Addresses

\end{document}